\documentclass{amsart}
\usepackage{youngtab}
\usepackage{graphicx,tikz}
\usepackage{mathrsfs,hyperref}
\usepackage{cancel}

\DeclareMathOperator{\spn}{span}

\DeclareMathOperator{\Hom}{Hom}

\numberwithin{equation}{section} 

\newcommand{\inv}{^{-1}}
\newcommand{\ol}{\overline}
\newcommand{\la}{\mathfrak}
\newcommand \C[1]{{\mathcal #1}}
\newcommand \Tn{\mathscr{T}_n}
\newcommand \T[1]{T_{s_{#1}}}
\newcommand \dT[1]{\dot{T}_{s_{#1}}}
\newcommand \kl{Kazhdan--Lusztig~}

\newcommand{\fyt}{f^\mathrm{YT}}
\newcommand{\fkl}{f^{\operatorname{KL}_\CC}}
\newcommand{\fw}{f^\mathrm{web}}
\newcommand{\dyt}{D^\mathrm{YT}}
\newcommand{\dkl}{D^{\operatorname{KL}_\CC}}
\newcommand{\dw}{D^\mathrm{web}}
\newcommand{\dsn}{D^{S_n}}
\newcommand{\fsn}{f^{S_n}}
\newcommand{\ten}{10}
\newcommand{\eleven}{11}
\newcommand{\twelve}{12}

\newcommand{\klc}{\operatorname{KL}_{\CC}}
\newcommand\bZ{\mathbb{Z}}
\newcommand\Hn{\mathscr{H}_n}
\newcommand\tab{Y}
\newcommand\CC{\mathcal C}
\newcommand\bC{\mathbb C}
\newcommand\dk[1]{\overset{#1}{\underset{dK}{\approx}}}
\newcommand\knuth[1]{\overset{#1}{\underset{K}{\approx}}}

\newtheorem{theorem}{Theorem}[section]
\newtheorem{question}[theorem]{Open Question}

\newtheorem{lemma}[theorem]{Lemma}

\newtheorem{corollary}[theorem]{Corollary}
\newtheorem{statement}[theorem]{Statement}

\theoremstyle{definition}
\newtheorem{definition}[theorem]{Definition}
\newtheorem{example}[theorem]{Example}
\newtheorem{observation}[theorem]{Observation}
\newtheorem{remark}[theorem]{Remark}

\everymath{\displaystyle}
\title{The Robinson--Schensted Correspondence and $A_2$-web Bases}

\author{Matthew Housley}
\address{Department of Mathematics, University of Utah, 155 S 1400 E Room 233, Salt Lake City, UT 84112 U.S.A.}
\email{housley@math.utah.edu}

\author{Heather M. Russell}
\address{Department of Mathematics and Computer Science, Washington College, 300 Washington Avenue, Chestertown, MD 21620 U.S.A.}
\email{hrussell2@washcoll.edu}

\author{Julianna Tymoczko}
\address{Department of Mathematics and Statistics, Smith College, Northampton, MA 01063 U.S.A.}
\email{jtymoczko@smith.edu}
\thanks{MH was supported by the National Science Foundation. HR was supported by the John Templeton Foundation.  JT was supported by a Sloan Fellowship and by National Science Foundation grants DMS-0801554 and DMS-1248171.}

\begin{document}

\begin{abstract}
We study natural bases for two constructions of the irreducible representation of the symmetric group corresponding to $[n,n,n]$: the {\em reduced web} basis associated to Kuperberg's combinatorial description of the spider category; and the {\em left cell basis} for the left cell construction of Kazhdan and Lusztig.  In the case of $[n,n]$, the spider category is the Temperley-Lieb category; reduced webs correspond to planar matchings, which are equivalent to left cell bases.  This paper compares the image of these bases under classical maps: the {\em Robinson--Schensted algorithm} between permutations and Young tableaux and {\em Khovanov--Kuperberg's bijection} between Young tableaux and reduced webs.  

One main result uses Vogan's generalized $\tau$-invariant to uncover a close structural relationship between the web basis and the left cell basis.    Intuitively, generalized $\tau$-invariants refine the data of the inversion set of a permutation.  We  define generalized $\tau$-invariants intrinsically for \kl left cell basis elements and for webs. We then show that the generalized $\tau$-invariant is preserved by these classical maps.  Thus, our result allows one to interpret Khovanov--Kuperberg's bijection as an analogue of the Robinson--Schensted correspondence.

Despite all of this, our second main result proves that the reduced web and left cell bases are inequivalent; that is, these bijections are not $S_{3n}$-equivariant maps. 
\end{abstract}

\maketitle

\section{Introduction}

This paper studies two important bases for the $[n,n,n]$ representation of the symmetric group $S_{3n}$---the {\em \kl basis} and the {\em reduced web basis}---to compare their images under  classical maps.   

Kuperberg defined the combinatorial $A_2$ spider \cite{K}, which comes with a braiding that induces a symmetric group action on webs.  The combinatorial $A_2$ spider is a diagrammatic category encoding the representation theory of $\C U_q(\la{sl}_3)$, namely the quantum enveloping algebra of the Lie algebra $\la{sl}(3,\bC)$.  The category comes with objects (tensor products of the standard representation $V^+$ and the dual representation $V^-$ of the quantum group), morphisms  (intertwining maps), and a diagrammatic description of morphisms compatible with the standard skein relations of knot theory.  The analogous object for $\la{sl}(2,\bC)$ is the $A_1$ spider or Temperley-Lieb category.  (Spiders for  $\C U_q(\la{sl}_n)$ have also been constructed \cite{Kim,MR2710589}.)  In the classical limit, the symmetric group acts on a tensor power by exchanging factors; this action is represented diagrammatically by twining certain strands of the web and resolving according to standard knot-theoretic relations.  In this paper, as in much other combinatorial work  \cite{PPR, T}, we focus on webs for $(V^+)^{\otimes3n}$.

The action of the symmetric group on the \kl basis comes from a totally different context: the construction of the Hecke algebra and action of the Hecke algebra on itself, which specializes to an action of the symmetric group on its group algebra in the classical limit.  Kazhdan--Lusztig's original work on the Hecke algebra  aimed to understand Springer's geometric representation of the symmetric group.  They described a basis they hoped was Springer's basis, and in the process discovered a different, more algebraically-natural basis of the Hecke algebra now called the \kl basis. The \kl basis yields \emph{left cell bases} for each irreducible representation of the symmetric group under a standard process described in Section \ref{klsection}.  A close relationship seems to exist between these \kl left cell bases and the Springer basis of each irreducible symmetric group representation (see below), though its exact nature is mysterious.

These bases of irreducible symmetric group representations tie together numerous areas of mathematics. The \kl left cell bases have important applications to problems in combinatorics \cite{MR2557880}, geometry \cite{MR1782635}, and the theory of infinite dimensional Lie algebra representations \cite{MR632980,MR610137}. The reduced web basis plays an important role in both category and knot theory \cite{K}. Our work is part of a wider effort to understand these bases and their various applications.

One traditional approach to studying bases is to identify equivalences by constructing isomorphisms.  For instance, Khovanov constructed isomorphisms sending Deodhar's relative \kl basis to two different bases:  Lusztig's canonical basis and dual canonical basis in the context of $\C U_q(\la{sl}_2)$ \cite{MR1446615}.  Khovanov defined his map algebraically, though he was inspired by  diagrammatic calculations in the Temperley-Lieb algebra.  Frenkel, Khovanov, and Kirillov extended Khovanov's result to an isomorphism mapping the relative \kl basis to Lusztig's standard and dual canonical bases for all $\C U_q(\la{sl}_k)$ \cite{MR1657524}. 

Another traditional approach is to compare the images of different bases under established maps.  For instance, the natural restriction of the dual canonical basis to the $A_1$ spider (discussed later in this introduction) coincides with the reduced web basis for $\C U_q(\la{sl}_2)$ \cite{MR1446615}.  Kuperberg conjectured that in the $\la {sl}_3$ case, the restriction of the dual canonical basis would agree with the reduced web basis \cite{K}, though this was later disproven by Khovanov and Kuperberg \cite{KK}.  More recently, Fontaine, Kamnitzer, and Kuperberg identify reduced webs for $A_2$ with the components of the $[n,n,n]$ Springer fiber using a map closely related to the geometric Satake correspondence \cite{FKK,FON}.

In this paper, we consider well-known combinatorial bijections relating {the index sets for} different bases of $S_{3n}$--representations:
\begin{enumerate}
\item[I.] permutations naturally index the \kl basis elements;
\item[II.] the {\em Robinson--Schensted} algorithm associates permutations bijectively to pairs of standard Young tableaux; and

\item[III.]  {\em Khovanov--Kuperberg's bijection} identifies reduced webs for $(V^+)^{\otimes3n}$ with standard Young tableaux of shape $[n,n,n]$.

\end{enumerate}

Robinson--Schensted's algorithm needs little introduction.  Decomposing the regular representation of the symmetric group $S_n$ into a sum of irreducible representations (with multiplicity) shows that the order of $S_n$ is equal to the number of same shape pairs of Young tableaux on $n$ boxes.  Robinson--Schensted makes this bijection explicit; it is used throughout combinatorics, representation theory, and computer science \cite{MR1464693,MR2133266,MR1824028}.  

Similarly,  Khovanov--Kuperberg's bijection \cite[Proposition 1]{KK} makes explicit Kuperberg's more general observation that reduced webs are equinumerous with certain dominant lattice paths \cite{K};  we use Petersen--Pylyavskyy--Rhoades's elegant description of the map, which replaces dominant lattice paths with Young tableaux, together with Tymoczko's description of the inverse  \cite{T}.  Russell recently extended the map to a bijection in the context of arbitrary tensor products of $V^+$ and $V^-$ and semistandard Young tableaux \cite{MR3119361}.  

The current paper proves two main claims in a series of results.  The first is that the three bijections I--III preserve deep structural properties, in a sense we make precise in a moment.  The second is a combinatorial proof that the reduced web and \kl left cell bases of $[n,n,n]$ symmetric group representations are inequivalent in spite of their structural commonalities. 

The structure that these bijections preserve is called the {\em generalized $\tau$-invariant}.  Vogan first defined generalized $\tau$-invariants in the context of $\C U(\la g)$-modules 
\cite{MR545215}.  For permutations, the generalized $\tau$-invariant can be thought of intuitively as a refinement of the data of the descent set.  More precisely, given information about how to construct a descent set together with a collection of functions $f_{i,j}$ that change the descent set in prescribed ways, we can follow different sequences of functions $f_{i,j}$ while retaining the information of the descents sets associated to each permutation encountered along the way. The generalized $\tau$-invariant is closely related to Assaf's notion of dual equivalent graphs, wherein vertices are labeled with descent sets and an edge connects two vertices when they are linked by an $f_{i,j}$-map \cite{MR2710706}.

We provide intrinsic definitions of the descent set in different contexts, which we call the {\em $\tau$-invariant} to emphasize our purpose.  The $\tau$-invariant is the collection of simple reflections $s_i$ for which: 
\begin{itemize}
\item {\bf (for a standard Young tableau)}  the number $i+1$ lies on a lower row than $i$ (and a special case of Vogan's definition \cite{MR545215});
\item {\bf (for a \kl basis element)} the reflection $s_i$ negates the \kl basis element; and
\item {\bf (for a reduced web)} boundary vertices $i$ and $i+1$ are joined by an interior vertex.
\end{itemize}
We then identify natural functions $f_{i,j}$ in each context, allowing us to recursively construct the generalized $\tau$-invariant in each case.  Amazingly, the $\tau$-invariant, the maps $f_{i,j}$, and thus the generalized $\tau$-invariants commute with the natural maps described above and as shown in Figure \ref{figure: commuting gen taus}.  

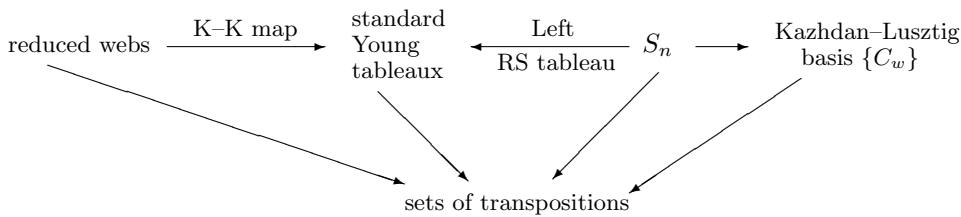
\begin{figure}[h]
\begin{picture}(370,80)(-20,-20)
\put(-20,40){\small reduced webs}
\put(110,50){\small standard}
\put(110,40){\small Young} 
\put(110,30){\small tableaux}
\put(220,40){$S_n$}
\put(270,45){\small \kl}
\put(280,35){\small basis $\{C_w\}$}

\put(40,42){\vector(1,0){60}}
\put(215,42){\vector(-1,0){60}}
\put(240,42){\vector(1,0){20}}

\put(178,46){\small Left}
\put(165,33){\small RS tableau}
\put(50,46){\small K--K map}

\put(130,-20){\small sets of transpositions}

\put(280,30){\vector(-3,-2){65}}
\put(226,32){\vector(-1,-1){40}}
\put(120,25){\vector(1,-1){34}}
\put(-5,35){\vector(3,-1){135}}
\end{picture}
\caption{Commuting diagram of various generalized $\tau$-invariants (diagonal arrows represent generalized $\tau$-invariants)}\label{figure: commuting gen taus}
\end{figure}

Vogan proved that the central triangle in  Figure \ref{figure: commuting gen taus} commutes \cite[Proposition 6.4]{MR523602}; we give an elementary and self-contained proof in Theorem \ref{theorem: generalized tau invariants for perms and RS}. Vogan's result is a direct and non-algorithmic characterization of the left tableau determined by the 
Robinson--Schensted correspondence.  Theorem \ref{klexchange} proves that the rightmost triangle in  Figure \ref{figure: commuting gen taus} commutes; in conjuction with the mapping from $S_n$ to left tableaux, this gives us an intrinsic version of Robinson--Schensted for \kl left cell basis elements.  This refines results of Bj{\"o}rner --- Theorems \ref{RSleftcell} and \ref{leftcellunique} --- parameterizing left cell bases by Young tableaux.

Our first main result, Lemma \ref{lemma: kupbij commutes with f}, treats the leftmost triangle in Figure \ref{figure: commuting gen taus}: it proves that Khovanov--Kuperberg's bijection commutes with the generalized $\tau$-invariant.  This means that Khovanov--Kuperberg's bijection, too, is a direct analogue of the Robinson--Schensted correspondence, as stated in Theorem \ref{maintheorem}.  Interestingly, these are not the only situations in which a generalized $\tau$-invariant arises: various results on orbital varieties and Young tableaux imply that the Springer basis for $\mathfrak{sl}_n$ also satisfies a generalized $\tau$-invariant property \cite{MR0447423,MR741942,MR733619,McGovern}.

Our second main result, Theorem \ref{thm:Inequivalence of bases}, shows that the reduced web and \kl left cell bases are not equivalent.  In other words, the linear map induced by sending each reduced web to its corresponding left cell basis element is not $S_{3n}$-equivariant. The fact that these two bases disagree is compatible with other results relating dual canonical, Kazhdan--Lusztig, and web bases. Frenkel, Khovanov, and Kuperberg showed that the dual canonical basis is equivalent to Deodhar's relative \kl basis \cite{MR1657524,MR916182} and  Khovanov and Kuperberg subsequently proved that the web and dual canonical bases are inequivalent. How Deodhar's  basis relates to the \kl left cell bases remains unknown, so our result is a suggestive complement to Khovanov and Kuperberg's work.  To the best of our knowledge this result appears nowhere in the literature, though some experts believe that it may also follow from Schur--Weyl duality.  Thus, Theorem  \ref{thm:Inequivalence of bases} confirms a piece of mathematical ``folklore."

There is a rich history of change-of-basis matrices that are upper-triangular with ones along the diagonal, together with the important polynomials that arise as entries in these matrices, including:  Kostka polynomials in the theory of symmetric functions \cite[Chapter 1.6]{MR1354144}, the \kl polynomials themselves (which essentially provide the change-of-basis between the \kl basis and the defining basis for the Hecke algebra), and the change-of-basis matrix between the Springer basis and the web basis \cite{FKK}. This leads us to the following open question.

\begin{question}
We conjecture that the change-of-basis matrix between the \kl basis and the basis of reduced webs is upper-triangular with ones along the diagonal.  What are the entries in the change-of-basis matrix between the \kl basis and the basis of reduced webs?  How do they relate to other combinatorial objects?
\end{question}

\section{\texorpdfstring{Generalized $\tau$-invariants for Tableaux}{Generalized tau-invariants for Tableaux}}\label{tableausection}

We begin by describing $\tau$-invariants and generalized $\tau$-invariants for tableaux.  In this context, the $\tau$-invariant of a standard tableau $Y$ is the set of pairs $i,i+1$ for which $i+1$ is in a row below $i$.  We will define processes that sequentially exchange certain pairs $i,i+1$ or $i+1,i+2$ in a tableau $Y$; the generalized $\tau$-invariant of $Y$ collects the $\tau$-invariants of the tableaux that result from all sequences of transpositions allowed in our processes.  Theorem \ref{gentautabthm}, the main theorem of this section, proves that if two tableaux with the same number of boxes have the same generalized $\tau$-invariant, then in fact they are the same---and in particular the tableaux have the same shape.

In this and subsequent sections, we will rely extensively on standard results from the combinatorics of tableaux and the symmetric group. Bj{\"o}rner and Brenti's book provides an excellent exposition of this material \cite{MR2133266}.  We quickly establish our conventions.

A \emph{Young diagram} is a collection of finitely many boxes arranged in top- and left-justified columns and rows. Young diagrams with $n$ boxes correspond naturally to partitions of $n$ by treating the length of each row as a part of a partition. A \emph{standard Young tableau} on a Young diagram with $n$ boxes is a labeling of the boxes with the numbers $1,2,\ldots, n$ in such a way that the labels increase strictly left-to-right and top-to-bottom.

We denote the set of all standard Young tableaux on $n$ boxes by $\Tn$. Let $s_i\in S_n$ be the simple transposition that exchanges $i$ and $i+1$.
\begin{definition}
Let $\tab\in \Tn$.  The \emph{$\tau$-invariant} of $\tab$ is the subset of simple transpositions $s_i$ in $S_n$ for which $i+1$ is below the row of $i$ in $\tab$. We denote the $\tau$-invariant of $\tab$ by $\tau(\tab)$. 
\end{definition}

Notions like that of the $\tau$-invariant are common when studying Young tableaux: in many of the standard ways to associate a permutation to a tableau, finding a larger number in a row below a smaller number corresponds to an inversion of the permutation.  Figures \ref{exampletableaux1} and \ref{exampletableaux2} show a number of tableaux with their $\tau$-invariants.

\begin{definition}
If $s_i,s_j$ are adjacent simple transpositions in $S_n$,   define $\dyt_{i,j}$ to be the set of all $\tab\in \Tn$ such that $s_i\in \tau(\tab)$ and $s_j\notin\tau(\tab)$. 
\end{definition}

Let $s_i \cdot \tab$ be the (not necessarily standard) tableau obtained from $\tab$ by exchanging $i$ and $i+1$. The proof of following fact is straightforward.

\begin{lemma}\label{fytlemma}
Let $s_i,s_j$ be adjacent simple transpositions in $S_n$. 
\begin{enumerate}
\item Given $\tab\in \dyt_{i,j}$ exactly one of $s_i\cdot \tab$ and $s_j\cdot \tab$ is a standard tableau in $\dyt_{j,i}$.  Denote this unique element 
by $\fyt_{i,j}(\tab)$.
\item The function $\fyt_{i,j}:\dyt_{i,j}\rightarrow \dyt_{j,i}$ is a bijection whose inverse is the function $\fyt_{j,i}:\dyt_{j,i}\rightarrow \dyt_{i,j}$.
 \end{enumerate}
\end{lemma}

Figures \ref{exampletableaux1} and \ref{exampletableaux2} give examples of $\fyt_{i,j}$ for different tableaux.

We now define the generalized $\tau$-invariant, one of the central definitions of this paper,  which will return in slightly different contexts in the next two sections.  The generalized $\tau$-invariant is constructed as an equivalence class on the collection of Young tableaux.  Intuitively, the generalized $\tau$-invariant records the data of the descents of $\tab$, together with all of the Young tableaux obtained as a sequence $\fyt_{i_1,j_1} \circ \fyt_{i_2,j_2} \circ \fyt_{i_3,j_3} \circ \cdots (\tab)$ where each function $\fyt_{i_k,j_k}$ is assumed to be well-defined on its input.  (In fact, as the reader will see, the generalized $\tau$-invariant contains more information.)  More precisely, we have the following.

\begin{definition}\label{gentaudef}
Let $\tab$ and $\tab'$ be elements of $\Tn$.
 If $\tau(\tab)=\tau(\tab')$, then we say that $\tab$ and $\tab'$ are \emph{equivalent to order $0$}, written $\tab\underset {0}{\approx}\tab'$. 
 We say that $\tab$ and $\tab'$ are {\em equivalent to order $n$} and write $\tab\underset {n}{\approx}\tab'$ if $\tab\underset {n-1}{\approx}\tab'$ and $\fyt_{i,j}(\tab)\underset {n-1}{\approx}\fyt_{i,j}(\tab')$ whenever $\tab$ and $\tab'$ are both in $\dyt_{i,j}$. 
 If $\tab\underset {n}{\approx}\tab'$ for all nonnegative integers $n$, then we say $\tab$ and $\tab'$ have the same \emph{generalized $\tau$-invariant} and write $\tau_g(\tab)=\tau_g(\tab')$.
\end{definition}

The data collected in Figures \ref{exampletableaux1} and \ref{exampletableaux2} demonstrates parts of the calculation of the generalized $\tau$-invariant for different tableaux.

\begin{remark}
Note that in \cite{MR545215}, Vogan defines a \emph{right} generalized $\tau$-invariant by using the right action of the symmetric group on itself. Our version uses the \emph{left} action instead. The right generalized $\tau$-invariant is simply $\tau_g(x\inv)$.
\end{remark}

\begin{remark}\label{remark: general gentaudef}
We defined the generalized $\tau$-invariant for tableaux $\tab \in \Tn$.  However, the definition can be extended to the elements $\tab$ of any set $\Omega$ as long as the set $\Omega$ is equipped with a $\tau$-invariant $\tau: \Omega \rightarrow \{\textup{subsets of }S_n\}$, subsets $D_{i,j} \subseteq \Omega$ and a family of functions $f_{i,j}: D_{i,j} \rightarrow D_{j,i}$.  In general, we use the subsets $D_{i,j} \subseteq \Omega$ consisting of the elements $\tab \in \Omega$ with $s_i \in \tau(\tab)$ and $s_j \notin \tau(\tab)$.  We will use this remark repeatedly in subsequent sections, where we define generalized $\tau$-invariants for permutations, \kl basis elements, and webs.
\end{remark}

A standard tableau $\Gamma$ is called \emph{column superstandard} if each column is labeled sequentially, starting with $1,2,3,\ldots$ in the first column and increasing by one between the bottom of one column and the top of the next column to the right. (The last tableau in Example \ref{colsup} has this form.)   This is a useful form for a tableau; it turns out that a judiciously chosen sequence of maps $\fyt_{i,j}$ will transform an arbitrary tableau into a column superstandard one.

\begin{lemma}\label{tableautosuperstandardlemma}
Given a standard tableau $\tab$, there exists a sequence of maps \linebreak $\fyt_{i_1,j_1}, \fyt_{i_2,j_2}, \fyt_{i_3,j_3},\ldots$ that carries $\tab$ to a column superstandard tableau.
\end{lemma}
\begin{example}\label{colsup}

\[
\raisebox{-.3in}{\young(125,34,6)}\ \underrightarrow{\scriptstyle\ \fyt_{5,4}\ }\ \raisebox{-.3in}{\young(126,34,5)}\ \underrightarrow{\scriptstyle\ \fyt_{4,3}\ }\ \raisebox{-.3in}{\young(126,35,4)}\ \underrightarrow{\scriptstyle\ \fyt_{2,1}\ }\ \raisebox{-.3in}{\young(136,25,4)}\ \underrightarrow{\scriptstyle\ \fyt_{3,4}\ }\
\raisebox{-.3in}{\young(146,25,3)}
\]

\end{example}
\begin{proof}[Proof of Lemma \ref{tableautosuperstandardlemma}]
The lemma is trivially true if $\tab$ has only one box, or even only one column. We  proceed by induction on the number of boxes in $\tab$.

Suppose $\tab$ has $n$ boxes and more than one column. In this case, there is a simple transposition $s_i\in S_n$ that is not in $\tau(\tab)$. Let $k$ be the largest integer such that $s_k\notin\tau(\tab)$ and assume that $k<n-1$ so that $\tab\in \dyt_{k+1,k}$.  Then $s_{k+1}\notin \tau(\fyt_{k+1,k}(\tab))$. A series of operations $\fyt_{k+1,k}(\tab)$ produces a tableau $\tab'$ such that $s_{n-1}\notin\tau(\tab')$. It follows that $n$ is not in the first column of $\tab'$.

First ignore the label $n$ in $Y'$.  The induction assumption allows us to put the remaining labels in column superstandard order. In particular, the first column of the resulting tableau is numbered sequentially $1,2,3,\ldots$. Now include $n$ but ignore the first column of this tableau. Shifting labels, the induction assumption allows us to put the remaining labels in column superstandard order. The resulting tableau is column superstandard.
\end{proof}

The next sequence of results builds towards a result of Vogan's: that the generalized $\tau$-invariant really is a complete invariant of tableaux, in the sense that different tableaux have different generalized $\tau$-invariants.  The first two lemmas treat special cases, while the subsequent theorem puts them together to prove the general claim.

\begin{lemma}\label{Yequalsgamma}
Let $\Gamma$ be a column superstandard tableau and $\tab$ a standard tableau of the same shape. If $\tau(\tab)=\tau(\Gamma)$ then $\tab=\Gamma$.
\end{lemma}
The proof is left to the reader.
\begin{lemma}
If $\tab, \tab'\in \mathscr{T}_n$ have different shapes then $\tau_g(\tab)\neq\tau_g(\tab')$.
\end{lemma}
\begin{proof}
Let $Y$ and $Y'$ be elements of $\Tn$.  We will compare the column lengths $c_1, c_2, \ldots$ of $Y$ to the column lengths $c_1', c_2', \ldots$ of $Y'$.  Suppose without loss of generality that the first column in which they differ is the $k^{th}$ column, where $c_k > c_k'$.  Let $\Gamma$ be the column superstandard tableau with the same shape as $Y$. Observe that all of $\{s_i: 1 \leq i \leq c_1 + c_2 + \cdots + c_k-1\}$ are in $\tau(\Gamma)$ except $s_{c_1}, s_{c_1+c_2}, \ldots, s_{c_ 1+c_2+\ldots}$, which correspond to the numbers at the bottom of each of the first $k$ columns.  

Let $Z$ be any tableau with the same shape as $Y'$.  We will show that $\tau(\Gamma) \neq \tau(Z)$.  Suppose  $\{s_1, s_2, \ldots, s_{c_1-1}\} \subseteq \tau(Z)$.  Then each of the integers $1, 2, 3, \ldots, c_1$ is in a lower row than its predecessor.  There are only $c_1$ rows in $Z$ so $1$ is in the first row, $2$ in the second, and so on.  Moreover, each of $1,2,\ldots,c_1$ must be in the first column since $Z$ is standard.  Repeat this argument to conclude that if $\tau(Z)$ contains all of $\{s_i: 1 \leq i \leq c_1 + c_2 + \cdots + c_k-1\}$ except $s_{c_1}, s_{c_1+c_2}, \ldots, s_{c_ 1+c_2+\ldots}$, then $Z$ must be column superstandard in the first $k-1$ columns and contain at least $c_k$ entries in the $k^{th}$ column.  This contradicts the hypothesis on the columns of $Y$ and $Y'$.

By Lemma \ref{tableautosuperstandardlemma}, there exists a sequence of maps $\fyt_{i_1,j_2}, \fyt_{i_2,j_2}, \fyt_{i_3,j_3},\ldots$ that carries $Y$ to $\Gamma$.  The image of $Y'$ under this sequence of maps is a tableau $Z$ for which $\tau(Z) \neq \tau(\Gamma)$.  It follows that $\tau_g(Y)\neq \tau_g(Y')$.
\end{proof}
\begin{theorem}[Vogan]\label{gentautabthm}
If $Y,Y'\in\Tn$ and $\tau_g(\tab)=\tau_g(\tab')$ then $\tab=\tab'$.
\end{theorem}
\begin{proof} We may assume that $\tab$ and $\tab'$ have the same shape by Lemma \ref{Yequalsgamma}.
Lemma \ref{tableautosuperstandardlemma} guarantees the existence of a sequence of maps $\fyt_{i_1,j_2}, \fyt_{i_2,j_2}, \fyt_{i_3,j_3},\ldots$ that take $\tab$ to a column superstandard tableau $\Gamma$.  The same sequence takes $\tab'$ to a standard tableau $\Gamma'$ satisfying $\tau(\Gamma')=\tau(\Gamma)$.  Lemma \ref{Yequalsgamma} implies $\Gamma=\Gamma'$. Because the $\fyt_{i,j}$-maps are invertible, we know $Y=Y'$.
\end{proof}

\section{Generalized \texorpdfstring{$\tau$}{tau}-invariants and the Robinson--Schensted Correspondence}\label{snsection}

In this section we define $\tau$-invariants and generalized $\tau$-invariants for elements of the symmetric group. Our exposition will follow the template of the previous section; the reader will notice more than stylistic similarities.  In fact, these $\tau$-invariants are closely related to those defined for Young tableaux: we prove in Theorem \ref{theorem: generalized tau invariants for perms and RS} that the generalized $\tau$-invariants of a permutation are precisely the generalized $\tau$-invariants for the left standard tableau corresponding to $w$ under the Robinson-Schensted algorithm. As before, standard results can be found in, e.g., Bj\"{o}rner-Brenti's work \cite{MR2133266}.

Take  the simple transpositions $s_1,s_2,\ldots, s_{n-1}$ as a generating set for $S_n$. The \emph{length} $\ell(x)$ of $x$  is the minimal number $k$ needed to write $x$ in terms of the generators $x = s_{i_1} s_{i_2} \cdots s_{i_k}$. Given $x,y\in S_n$  we say $x\leq y$ if $x$ is a (not necessarily connected) subword of a minimal-length expression for $y$ in terms of the generators $s_1, s_2, \ldots, s_{n-1}$. The partial order $\leq$ is known as the Bruhat order.

\begin{definition} Let $x \in S_n$.  The $\tau$-invariant of $x$ is the set $\tau(x)$ consisting of simple transpositions $s_i$ such that $s_i x<x$.
\end{definition}

The set $\tau(x)$ is often called the {\em descent set} of the permutation $x^{-1}$, or the \emph{left descent set} of $x$ in the terminology of \cite{MR2133266}.  

We use one-line notation for permutations as well as factorizations into simple reflections.  Recall that the one-line notation for $x$ is the sequence $x_1x_2x_3\cdots x_n$ where $x$ sends $1$ to $x_1$, $2$ to $x_2$, etc.  The $\tau$-invariant of $x$ is easy to read in one-line notation: in fact $s_i \in \tau(x)$ if and only if $i+1$ appears somewhere to the left of $i$ in the one-line notation for $x$.

The next definition gives an analogue of $\dyt_{i,j}$ for permutations.

\begin{definition}
If $s_i$ and $s_j$ are adjacent simple transpositions, let $\dsn_{i,j}$ be the set of all $x\in S_n$ such that $s_i\in\tau(x)$ and $s_j\notin\tau(x)$. 
\end{definition}

To construct an appropriate analogue of the map $\fyt_{i,j}$ we must introduce \emph{dual Knuth relations}.

\begin{definition}
Two elements $x,y\in S_n$ are said to be related by a \emph{dual Knuth step of type $i$} if their one-line notations differ: 
\begin{itemize}
\item by transposing $i$ and $i+1$ and if $i-1$ appears between $i$ and $i+1$; or 
\item by transposing $i-1$ and $i$ and if  $i+1$ appears between $i-1$ and $i$. 
\end{itemize}
We denote this by $x\dk{i}y$. If two symmetric group elements are related by a sequence of dual Knuth steps, then we say that $x$ and $y$ are \emph{dual Knuth equivalent}.
\end{definition}
\begin{example}
\[25413\dk{2} 35412\]
\end{example}
\begin{remark}\label{remark: relationship between Knuth and dual Knuth}
The literature often treats Knuth steps $ \knuth{i}$ as the primary relation and dual Knuth steps $\dk{i}$ as the subsidiary.  The two are dual in the sense that $x\dk i y$ if and only if $x\inv\knuth i y\inv$.  Knuth steps are stated in terms of {\em positions} of the permutation and correspond to multiplication on the right by $s_i$ or $s_{i-1}$, while dual Knuth steps treat {\em values} of the permutation and correspond to left-multiplication by $s_i$ or $s_{i-1}$.  
\end{remark}
Dual Knuth steps appear naturally in our applications, as follows.
\begin{remark}\label{remark: descents and dual Knuths}
If $x\dk i y$ then $\tau(x)$ and $\tau(y)$ each contain exactly one of $s_{i-1},s_i$.  In fact $x\in\dsn_{i-1,i}$ and $y\in\dsn_{i,i-1}$ or vice versa. The converse also holds, so given $x\in S_n$ there exists $y\in S_n$ satisfying $x\dk i y$ if and only if $x\in\dsn_{i-1,i}$ or $x\in\dsn_{i,i-1}$. 
\end{remark}
The previous remark leads us to define an analogue of $\fyt_{i,j}$ for permutations.
\begin{definition}\label{definition: fsn_{i,j}}
Let $s_i$ and $s_j$ be adjacent simple transpositions and $x$ an element of $\dsn_{i,j}$. Then exactly one of $s_i x$ and $s_j x$  is an element of $\dsn_{j,i}$.  We denote this unique element $\fsn_{i,j}(x)$.
\end{definition}
Remark \ref{remark: descents and dual Knuths} explained that $x\dk{k}\fsn_{i,j}(x)$, where $k$ is the larger of $i$ and $j$.

The next lemma is analogous to Lemma \ref{fytlemma}; the proof is left to the reader.

\begin{lemma}
The map $\fsn_{i,j}\colon\dsn_{i,j}\rightarrow \dsn_{j,i}$ is a bijection with inverse \linebreak $\fsn_{j,i}\colon\dsn_{j,i}\rightarrow \dsn_{i,j}$.
\end{lemma}

We may now define a generalized $\tau$-invariant for symmetric group elements by making the appropriate substitutions in Definition \ref{gentaudef}, as described in Remark \ref{remark: general gentaudef}.

The well-known \emph{Robinson-Schensted correspondence} gives a bijection between elements of $S_n$ and the set of ordered pairs of same-shape, standard Young tableaux with $n$ boxes. We sketch the algorithm in an extended example and refer the reader to the literature for details \cite[Section A3.3]{MR2133266}. Given $w\in S_n$ let $P(w)$ and $Q(w)$ denote respectively the left and right tableaux in the pair corresponding to $w$.
\begin{example} Let $w=54312$. We proceed left-to-right through the one-line notation for $w$. Each time, we add the next number of $w$ into the first row of the left tableau and we record where a new box was added in the right tableau. $P_i$ is the insertion tableau and $Q_i$ is the recording tableau at the $i^{th}$ step. 
\begin{itemize}
\item[54312:] Begin with the empty tableau for both $P$ and $Q$. $$(P_0, Q_0) = \left( \emptyset, \emptyset \right)$$
\item[\cancel{5}4312:] The insertion tableau gets the first number of $w$. The recording tableau gets a 1.     $$(P_1, Q_1) = \left( \, \raisebox{-3pt}{\young(5)} \, , \,  \raisebox{-3pt}{\young(1)} \, \right)$$
\item[\cancel{54}312:] Since $4<5$ the top entry is bumped down in the insertion tableau. The recording tableau notes the second box to be added.
$$ \hspace{.2in}  (P_2, Q_2) = \left( \, \raisebox{-7pt}{\young(4,5)} \, , \,  \raisebox{-7pt}{\young(1,2)} \, \right)$$
\item[\cancel{543}12:]   Since $3<4$, the top entry is bumped to the second row which in turn bumps that entry to the third row in the insertion tableau. The recording tableau notes the third box to be added.

$$(P_3, Q_3) = \left(\, \raisebox{-14pt}{\young(3,4,5)} \, , \, \raisebox{-14pt}{ \young(1,2,3)} \, \right)$$
\item[\cancel{5431}2:] Since $1<3$, the top entry is bumped to the second row, where it bumps that entry to the third row, where it bumps that entry to the fourth row of the insertion tableau.  The recording tableau notes the most recent box to be added.
$$(P_4, Q_4) = \left(\, \raisebox{-19pt}{\young(1,3,4,5)} \, , \, \raisebox{-19pt}{ \young(1,2,3,4)} \, \right)$$
\item[\cancel{54312}] Since $2>1$, we place it in the first row in the insertion tableau without bumping anything. The recording tableau notes where the last box was added.
$$(P_5, Q_5) = (P(w), Q(w)) =  \left(\, \raisebox{-19pt}{\young(12,3,4,5)} \, , \, \raisebox{-19pt}{ \young(15,2,3,4)} \, \right)$$
\end{itemize}
\end{example}
The next lemma collects classical results about how the Robinson-Schensted algorithm is affected by inverting a permutation and by Knuth equivalence, respectively. 
\begin{lemma}\label{lemma: Knuth results} Let $x,y \in S_n$.

\begin{itemize}
\item \cite[Fact A3.9.1]{MR2133266} Left and right tableaux satisfy $P(x)=Q(x\inv)$.

\item \cite[Lemma 6.4.4]{MR2133266} If $x\knuth{i}y$ then:
\begin{enumerate}
\item $P(x)=P(y)$.
\item $Q(y)$ is equal either to $s_{i-1}\cdot Q(x)$ or $s_i\cdot Q(x)$.
\end{enumerate}
\end{itemize}
\end{lemma}
We can combine the previous properties with facts about dual Knuth steps to obtain a dual Knuth analogue of the previous lemma.  We state the result in terms of the function $\fsn_{i,j}$ to minimize notation.
\begin{lemma}\label{Knuthrelationlemma1}
Let $x,y\in S_n$ with $x=\fsn_{i,j}(y)$. Then:
\begin{enumerate}
\item $Q(x)=Q(y).$
\item $P(y)$ is equal either to $s_{i}\cdot P(x)$ or $s_j\cdot P(x)$.
\end{enumerate}
\end{lemma}

\begin{proof}
By definition, the relation $x=\fsn_{i,j}(y)$ implies that $x \dk k y$, where $k$ is the larger of $i,j$.  Remark \ref{remark: relationship between Knuth and dual Knuth} described the relationship between Knuth and dual Knuth steps, namely that $x \dk k y$ if and only if $x^{-1} \knuth{k} y^{-1}$.  Inverting a permutation corresponds to exchanging the recording and insertion tableaux by Lemma \ref{lemma: Knuth results}.  We conclude that $Q(x) = Q(y)$ and that $P(y)$ equals either $s_i \cdot P(x)$ or $s_j \cdot P(x)$, again using Lemma \ref{lemma: Knuth results}.
\end{proof}

The next lemma shows that the $\tau$-invariants of $x$ and the insertion tableau $P(x)$ are the same; the proof tracks the Robinson--Schensted algorithm carefully.

\begin{lemma}\label{lemma: descents for YT and perms}
Let $x\in S_n$.  The transposition $s_i\in \tau(x)$ if and only $s_i\in \tau(P(x))$.
\end{lemma}

\begin{proof}
Consider $x$ in one-line notation.  The Robinson--Schensted algorithm successively inserts integers into the first row of the insertion tableau.  If $i$ is to the right of $i+1$, then either $i+1$ always stays on a row below that of $i$, or we attempt to insert $i$ into the row containing $i+1$ and bump $i+1$ to a lower row.  In both cases $i+1$ ends on a strictly lower row of $P(x)$ than $i$.  Conversely, suppose $i+1$ is to the right of $i$.  Inserting $i+1$ in the row that contains $i$ either bumps a larger number down or increases the length of that row.  If $i$ and $i+1$ are on the same row then $i$ will be bumped before $i+1$.  So at each stage of the Robinson--Schensted algorithm, the integer $i$ remains on the same row or below the row of $i+1$.
\end{proof}

The next corollary combines Lemma  \ref{Knuthrelationlemma1} with the argument from Lemma \ref{lemma: descents for YT and perms} to show that the recording tableau is invariant under the map $\fsn_{i,j}$ and the insertion tableau commutes with the map $\fsn_{i,j}$, in the sense that $P \circ \fsn_{i,j} = \fyt_{i,j} \circ P$.

\begin{corollary}\label{corollary: RS and the f maps}
Let $x\in \dsn_{i,j}$.  Then
\begin{enumerate}
\item $Q(\fsn_{i,j}(x))=Q(x)$.
\item $P(\fsn_{i,j}(x))=\fyt_{i,j}(P(x))$.
\end{enumerate}
\end{corollary}

\begin{proof}
Part (1) follows directly from Part (1) of Lemma \ref{Knuthrelationlemma1}.  For Part (2), we apply Part (2) of Lemma \ref{Knuthrelationlemma1}, which says that $P(\fsn_{i,j}(x))$ is either $s_i\cdot P(x)$ or $s_j\cdot P(x)$.  Lemma \ref{lemma: descents for YT and perms} proves that $P(\fsn_{i,j}(x))$ and $x$ have the same $\tau$-invariants, so in particular $P(\fsn_{i,j}(x))\in \dsn_{j,i}$. Since $P(\fsn_{i,j}(x))$ is a standard tableau, Lemma \ref{fytlemma} implies that $P(\fsn_{i,j}(x))=\fyt_{i,j}(P(x))$ as desired.
\end{proof}

We have now built to one of the main claims of this section:  the generalized $\tau$-invariant of a permutation $x$ is the same as the generalized $\tau$-invariant of the insertion tableau $P(x)$.  In particular, if two permutations have the same generalized $\tau$-invariant then so do their corresponding insertion tableaux.  While the claim is important, the proof is a straightforward induction using previous results.

\begin{theorem}\label{theorem: generalized tau invariants for perms and RS}
Given $x,y\in S_n$, the generalized $\tau$-invariant $\tau_g(x)=\tau_g(P(x))$.  In particular we have $\tau_g(x)=\tau_g(y)$ if and only if $\tau_g(P(x))=\tau_g(P(y))$.
\end{theorem}

\begin{proof}
Induct using Lemma \ref{lemma: descents for YT and perms} as the base case and Part (2) of Corollary \ref{corollary: RS and the f maps} as the inductive hypothesis.
\end{proof}

Finally, we come to the main claim of this section, which was originally proven by Vogan \cite{MR545215}:  the generalized $\tau$-invariant is preserved by the Robinson-Schensted algorithm, in the strong sense that the insertion tableau $P(x)$ is the {\em only} tableau whose generalized $\tau$-invariant is $\tau_g(x)$.  This means our definitions of generalized $\tau$-invariants for Young tableaux and permutations are intrinsic and natural.  

\begin{theorem}[The Robinson--Schensted Correspondence via $\tau_g$]\label{RSCtheorem}
Let $x\in S_n$.  The insertion tableau $P(x)$ is the unique element of $\Tn$ such that $\tau_g(P(x))=\tau_g(x)$, while $Q(x)$ is the unique tableau in $\Tn$ such that $\tau_g(Q(x))=\tau_g(x\inv)$.
\end{theorem}

\begin{proof}
Suppose $Y$ is an element of $\Tn$ with $\tau_g(Y)=\tau_g(x)$.  Then $\tau_g(Y)=\tau_g(P(x))$ by Theorem \ref{theorem: generalized tau invariants for perms and RS}, and so by Theorem \ref{gentautabthm} we have $Y = P(x)$.  Similarly, suppose $Y'$ is an element of $\Tn$ with $\tau_g(Y')=\tau_g(x^{-1})$.  We again know  $\tau_g(x^{-1})=\tau_g(P(x^{-1}))$ by Theorem \ref{theorem: generalized tau invariants for perms and RS}, and $\tau_g(P(x^{-1})) = \tau_g(Q(x))$ by Lemma \ref{lemma: Knuth results}.  It follows that $\tau_g(Y') = \tau_g(Q(x))$ and hence by Theorem \ref{gentautabthm} we conclude $Y' = Q(x)$.
\end{proof}

\section{Kazhdan--Lusztig Theory}\label{klsection}
In the previous sections, we described the generalized $\tau$-invariant for permutations and for standard tableaux, which collects the data of certain chains of inversions.  We showed that the generalized $\tau$-invariant is in fact a complete invariant, in the sense that  the generalized $\tau$-invariant uniquely identifies the permutation or standard tableau, respectively.  We then showed that the classical Robinson-Schensted correspondence between permutations and (pairs of) standard tableaux preserves the generalized $\tau$-invariant, so these different descriptions are essentially the same.

Standard tableaux are often used to construct irreducible representations of the symmetric group.  In this section, we review the essential parts of Kazhdan-Lusztig theory \cite{MR560412}, which gives an alternate construction of irreducible representations of the symmetric group, via representations of its associated Hecke algebra.  We then recall results of Bj\"{o}rner's that relate the {\em left cells} of Kazhdan-Lusztig theory to the standard tableaux corresponding to a particular irreducible representation.

We begin with the classical description of the Hecke algebra of $S_n$, though we use the deformation variable $v$ so that $q$ can be used later.  (Section \ref{webandkl} gives a  different presentation of the Hecke algebra that is more convenient in that context.)
\begin{definition}
Let $A$ be the ring $\bZ[v^{1/2},v^{-1/2}]$. The \emph{Hecke algebra} $\mathscr{H}_n$ of the symmetric group $S_n$ is the associative $A$-algebra with generators $\T{1},\T{2},\ldots, \T{n-1}$ and relations
\begin{equation}\T{i}\T{j}=\T{j}\T{i}\text{ if }\lvert i-j\rvert>1\end{equation}
\begin{equation}\T{i}\T{j}\T{i}=\T{j}\T{i}\T{j}\text{ if }\lvert i-j\rvert=1\end{equation}
\begin{equation}(\T{i}+1)(\T{i}-v)=0.
\label{heckerelation}
\end{equation}
Given $w\in S_n$ and any reduced expression $w = s_{i_1}s_{i_2}\ldots,s_{i_k}$ in terms of the simple transpositions, we write
$
T_{w}=\T{i_1}\T{i_2}\ldots\T{i_k}
$. One can check that the elements $T_w$ are well-defined.
\end{definition}
When we evaluate at $v=1$, the Hecke algebra $\Hn$ reduces to the group algebra of $S_n$ over $\bZ$. 

Define an involution $a\mapsto\overline{a}$ of $A$ by 
letting $\ol{v^{1/2}}=v^{-1/2}$ and extending linearly to the rest of $A$. We define an involution of $\mathscr{H}_n$ by 
\[
\ol{a T_w}=\ol a\, T_{w\inv}\inv.
\]
(Equation \eqref{heckerelation} can be rewritten as an equality between the unit $v \in A$ and a multiple of $T_{s_i}$.  Thus $T_{s_i}$ and hence $T_w$ are units in $\mathscr{H}_n$, as implied in this involution.)

The first main theorem in Kazhdan-Lusztig theory constructs a unique basis whose elements $C_w$ are (1) invariant under the involution on $\mathscr{H}_n$ and (2) a sum $\sum_{y \leq w} q_{y,w} T_y$ for Laurent polynomials $q_{y,w}$ in $v^{1/2}$.
\begin{theorem}[Kazhdan--Lusztig]
For any $w\in S_n$ there is a unique element $C_w\in \Hn$ such that 
\begin{itemize}
\item[] $\ol{C_w}=C_w$
\item[] 
$C_w=v^{\ell(w)/2}\sum_{y\leq w} (-1)^{\ell(w)-\ell(y)}v^{-\ell(y)}\,\ol{P_{y,w}}\, T_y$
\end{itemize}
where $P_{y,w}\in A$ is a polynomial in $v$ of degree less than or equal to $(\ell(w)-\ell(y)-1)/2$ for $y<w$ and $P_{w,w}=1$. Furthermore $\{C_w\}_{w\in S_n}$ is a basis of $\Hn$. 
\end{theorem}
The basis $\{C_w\}$ is the \emph{\kl basis}, and the polynomials $P_{y,w}$ are referred to as \emph{\kl polynomials.}  The \kl basis and \kl polynomials have been intensely studied by combinatorists and representation theorists alike.  

The following weighted graph gives a convenient shorthand for describing the action of $\T{i}$ on $\Hn$.  The graph will also be useful in the next section, where we will show that it is closely related to $\fsn_{i,j}$ and $\dsn_{i,j}$.
\begin{definition}
The \emph{Kazhdan--Lusztig graph} of $\Hn$ is the undirected graph whose vertices are labeled by the elements of $S_n$.  
Let $y,w\in S_n$ with $y\leq w$.  The multiplicity of the edge between $y$ and $w$ is given by
\[
\mu(y,w)=\begin{cases}
\text{the leading coefficient of }P_{y,w} \mbox{ if }\operatorname{deg}(P_{y,w})=(\ell(w)-\ell(y)-1)/2
\\
0 \mbox{ otherwise.}
\end{cases}
\]
In addition, define $\mu(w,y)=\mu(y,w)$. 
\end{definition}

\kl polynomials are in $A$ so the multiplicities $\mu(w,y)$ are integers. Note that $\mu(w,y) = 0$ means that there is no edge between $y$ and $w$.
 
The left action of a generator $\T{i}$ on $\Hn$ in terms of the basis $\{C_w\}_{w\in S_n}$ can be written explicitly using the function $\mu(w,y)$:
\begin{equation}\label{klaction}
\T{i} C_w=
\begin{cases}
-C_w&\mbox{if } s_i\in \tau(w), \textup{ and}\\
\\
v C_w+v^{1/2}\sum_{\substack{y\in S_n\\ 
s_i\in\tau(y)}}
\mu(w,y) C_y&\mbox{otherwise.}
\end{cases}
\end{equation}

Treating $\Hn$ itself as a left $\Hn$-module, we wish to decompose $\Hn$ into irreducibles. As it turns out, the natural irreducible objects will be quotients of certain subspaces of $\Hn$.  To define them, we need to understand \emph{left cells}.

\begin{definition}\label{Definition: left cell}
Define a binary relation $\preceq_L$ on $S_n$ by letting $x\preceq_L x$ and $x\preceq_L y$ whenever $\mu(x,y)\neq 0$ and $\tau(x)\not\subset\tau(y)$.  Extend $\preceq_L$ to a preorder by imposing transitivity. We refer to $\preceq_L$ as the \emph{left preorder on $S_n$}.
\end{definition}

The conditions defining $x \preceq_L y$ are equivalent to saying there is an $i$ for which $C_x$ appears as a summand in $\T{i} C_y$.

Our definition of $\preceq_L$ follows Bj\"{o}rner-Brenti's \cite{MR2133266}; Kazhdan-Lusztig used $\prec$ slightly differently in their original paper \cite{MR560412}.

\begin{definition}
Define an equivalence relation $\sim_L$ on $S_n$ by letting $x\sim_L y$ if $x\preceq_L y$ and $y\preceq_L x$. The equivalence classes under $\sim_L$ are called \emph{left cells} and the preorder $\preceq_L$ descends to a partial order on left cells.
\end{definition}

We can now construct \emph{left cell modules} and \emph{left cell representations}.
Let $\operatorname{Cell}(S_n)$ denote the set of left cells in $S_n$ ordered by $\preceq_L$. Given $\mathcal C\in\operatorname{Cell}(S_n)$ define
\[\spn_A{\mathcal C}=\spn_A\{C_w\vert w\in\mathcal C\}.\]
\begin{definition}
Let $\mathcal C\in \operatorname{Cell}(S_n)$.  The \emph{left cell module} $\operatorname{KL}_\mathcal C$ for $\mathcal C$ is defined as
\[
\operatorname{KL}_\mathcal C=\Biggl(\ 
\bigoplus_{\substack{\mathcal C_i\in \operatorname{Cell}(S_n)\\ \mathcal C_i\preceq_L\mathcal C}}
\spn_A \CC_i\Biggr)\Bigg/\Biggl(\ \bigoplus_{\substack{\mathcal C_i\in \operatorname{Cell}(S_n)\\ \mathcal C_i\prec_L\mathcal C}}
\spn_A \CC_i\Biggr).
\]
\end{definition}

For each $w \in \CC$ we often consider the image of $C_w$ under the natural projection to $\operatorname{KL}_\mathcal C$, and also denote this image by $C_w$.  We can explicitly compute the action of $\T{i}$ on the left cell module $\operatorname{KL}_{\mathcal C}$ by restriction.  Explicitly, if $C_w \in \operatorname{KL}_\mathcal C$ then we simply modify Equation \eqref{klaction}:
\begin{equation}\label{klcellaction}
\T{i} C_w=
\begin{cases}
-C_w\text{ if } s_i\in \tau(w),\\
\\
v C_w+v^{1/2}\sum_{\substack{y\in \CC\\ 
s_i\in\tau(y)}}
\mu(w,y) C_y\text{ otherwise.}
\end{cases}
\end{equation}
\begin{definition}
Consider the complex vector space obtained from $\operatorname{KL}_\mathcal C$ by  evaluating at $v=1$ and extending scalars from $\mathbb{Z}$ to $\bC$. Under the action of $S_n$ given by Equation \ref{klcellaction}, this vector space is called the \emph{left cell representation} derived from the cell $\CC$ and also  denoted $\operatorname{KL}_\mathcal C$.
The \emph{left cell graph} of $\mathcal{C}$ is the restriction of the \kl graph to $\mathcal{C}$.
\end{definition}

Recall that the number of irreducible representations of $S_n$  
equals the number of partitions of $n$. The following bijection from representations to partitions is due to Young; other classical constructions of the irreducible representations yield equivalent maps. (See \cite[Theorem 10.1.1]{MR1251060} and the discussion in the beginning of Section 6.5 of \cite{MR2133266}.)
\begin{theorem}\label{thm:snrepresentations}
 Let $\mathbf{p}=[p_1,p_2,\ldots]$ be a partition of $n$ and $\mathbf t=[t_1,t_2,\ldots]$ its transpose. There exists a unique irreducible $S_n$ representations $\pi_{\mathbf p}$ whose restriction to $\Pi_{p_i\in \mathbf p}S_{p_i}$ contains the trivial representation and whose restriction to $\Pi_{t_i\in \mathbf t}S_{t_i}$ contains the sign representation.
\end{theorem}
We  often refer to an irreducible $S_n$-representation via the Young diagram whose row-lengths describe the corresponding partition.

This parameterization of  irreducible representations of $S_n$ is closely related both to the Robinson--Schensted correspondence and to the decomposition of $S_n$ into left cells. As in Section \ref{snsection}, let $P(w)$ and $Q(w)$ denote the left and right tableaux respectively in the pair corresponding to $w\in S_n$. 

The following two theorems extend results of Kazhdan-Lusztig \cite[Theorem 1.4]{MR560412}. They were first published by Garsia-McLarnan \cite{MR937317}, who attributed the proofs to Bj\"{o}rner; Bj\"{o}rner-Brenti's book also has elegant proofs \cite[Theorems 6.5.1-3, pp. 189-190]{MR2133266}.

The first theorem says that  \kl left cells can also be described as the permutations that share a fixed right tableaux under the Robinson-Schensted algorithm.  Moreover, the left cell representation is exactly the irreducible representation corresponding to the shape of the right tableau.
\begin{theorem}[Bj\"{o}rner] \label{RSleftcell}
For each left cell $\CC$ of $S_n$ there exists a standard Young tableau $Q_{\CC}$ with $n$ boxes such that
\[
\CC=\{w\in S_n\vert Q(w)=Q_{\CC}\}.
\]
Furthermore $\operatorname{KL}_\mathcal C$ is isomorphic as a complex representation  to the irreducible representation parameterized by the shape of $Q_{\CC}$. 
\end{theorem}
Now suppose that $Q$ and $Q'$ are $n$-box standard Young tableaux of the same shape, and that $Q$ and $Q'$ are the right tableaux corresponding to left cells $\CC_Q$ and $\CC_{Q'}$ respectively. Define a map $\phi_{Q,Q'}:\CC_Q\rightarrow\CC_{Q'}$ by
\[
(P,Q)\mapsto (P,Q').
\]
\begin{theorem}[Kazhdan--Lusztig, Bj\"{o}rner]\label{leftcellunique}
The map $\phi_{Q,Q'}$ is a well-defined isomorphism of left cell graphs.
\end{theorem}
In particular, the map $\phi_{Q,Q'}$ preserves the edge-weights of the left cell graphs.

A priori, Theorem \ref{RSleftcell} left us with a collection of left cells that correspond to the same irreducible $S_n$-representation.  Theorem \ref{leftcellunique} says that these left cells are canonically isomorphic, and that their associated based representations are isomorphic; furthermore, the isomorphisms preserve left tableaux.  All the information in Equation \eqref{klcellaction} is preserved by the isomorphism $\phi_{Q,Q'}$, so we also refer to {\em the} \kl left cell basis for an irreducible $S_n$-representation  
without specifying a  left cell.

\section{An \emph{in situ} Robinson--Schensted Algorithm for Left Cells}\label{insitu}

Suppose that we are given an irreducible representation of $S_n$ 
with its associated \kl left cell basis but without the tableaux attached to each basis element. Is it possible to compute these tableaux directly by looking at the $S_n$-action on the basis elements? With an appropriate definition of the generalized $\tau$-invariant, the answer is yes.  We develop an appropriate definition in this section. The same definition will later allow us to reinterpret Khovanov--Kuperberg's map from webs to tableaux using the combinatorics of the symmetric group.

The ideas in this section follow from Vogan's work \cite{MR545215} in combination with the (now proven) \kl conjectures \cite{MR560412,MR632980,MR610137}, but we will proceed in a much more elementary fashion.

We recall the following fact.
\begin{lemma}\label{Snbruhatedgelemma} Let $x, s_i \in S_n$ with $s_i$ a simple transposition. Then 
$\mu(x,s_i x)=1$ in the \kl graph.
\end{lemma}

\begin{proof}
Kazhdan-Lusztig proved that if $s_i$ is not in $\tau(x)$ then $\mu(s_ix,x)=1$ \cite[p.~171, equation 2.3.a]{MR560412}.  The function $\mu(w,y)$ is symmetric so $\mu(x,s_i x)=1$ as well.  Recall that $\tau(x)$ consists of simple transpositions $s_j$ such that $s_j x<x$; thus, $s_i$ is in exactly one of $\tau(x)$ and $\tau(s_i x)$.  This proves the claim.
\end{proof}

An edge of the Kazhdan-Lusztig graph that connects vertices $x$ and $s_i x$ is often called a \emph{Bruhat edge} \cite[p.~176]{MR2133266}.  

The next theorem ties together ideas from several previous sections: first, each permutation $x$ in $\dsn_{i,j}$ has exactly one edge to $\dsn_{j,i}$ in the Kazhdan-Lusztig graph; second, that unique edge is to the  permutation $\fsn_{i,j}(x)$ and is labeled $1$.

\begin{theorem}\label{klexchange}
Let $x\in \dsn_{i,j}$. Then $x$ is connected to $y=\fsn_{i,j}(x)$ by an edge labeled $\mu(x,y)=1$.  Furthermore $x$ and $y$ lie in the same left cell and if $z\neq y$ is any other element of $\dsn_{j,i}$ then $\mu(x,z)=0$.
\end{theorem}
\begin{proof}
The function $\fsn_{i,j}(x)$ satisfies $y=s_i x$ or $s_j x$ by Definition \ref{definition: fsn_{i,j}}. Lemma \ref{Snbruhatedgelemma} then implies that $\mu(x,y)=1$. Since $\tau(x)\not\subset\tau(y)$ and $\tau(y)\not\subset\tau(x)$ we know $x$ and $y$ lie in the same left cell. 

Now suppose that $z\in \dsn_{j,i}$ with $\mu(x,z)\neq 0$. Statement (2.3.e) on page 171 of \cite{MR560412} immediately implies the following.
\begin{statement}Suppose $a, b\in S_n$ satisfy
both $s_k\notin \tau(a)$ and $s_k\in\tau(b)$. Then $b=s_k a$ if and only if $a<b$ and $\mu(a,b)\neq 0$. 
\end{statement}
Because $\mu(x,z)\neq 0$ either $x<z$ or $z<x$. If $x<z$ then $z= s_j  x=y$. If $z<x$ then $z=s_i  x=y$. 

\end{proof}

Let $\CC\subset S_n$ be an arbitrary left cell. As before,  the associated (complex) left cell representation is $\klc$ with basis $B(\klc)=\{C_w\vert w\in\CC\}$. We  want to define a $\tau$-invariant for each basis element $C_w$ intrinsically, using the $S_n$-action on $\klc$ of Equation \eqref{klcellaction}.  As with tableaux and permutations, we can define a set $\dkl_{i,j}$ for left-cell basis elements using this definition of $\tau$-invariants.
\begin{definition}
Let $C\in B(\klc)$ be a basis element for $\klc$. Define the $\tau$-invariant of $C$ to be
\[
\tau(C)=\{s_i\in S_n\vert \T{i} C=-C\}.
\]
Given adjacent simple transpositions $s_i, s_j$ let $\dkl_{i,j}$ be the set of basis elements $C$ in $\klc$ such that $s_i\in \tau(C)$ and $s_j\notin\tau(C)$.
\end{definition}

\begin{observation}\label{taubasispermutation}
By construction $\tau(C_w)=\tau(w)$ as desired.
\end{observation}

We would like to define $\fkl_{i,j}$ intrinsically rather than relying on the definition of $\fsn_{i,j}$.  Theorem \ref{klexchange} does this: it says that if $C_w\in \dkl_{i,j}$ and $\T{j} C_w$ is written in terms of the left-cell basis of $\klc$, then exactly one of the summands in $\T{j} C_w$  is an element of $\dkl_{j,i}$.  With this definition of $\fkl_{i,j}$ we will immediately be able to construct generalized $\tau$-invariants by substituting $\fkl_{i,j}$ in Definition \ref{gentaudef}, as described in Remark \ref{remark: general gentaudef}.
\begin{definition}
Let $C\in \dkl_{i,j}$.  Define $\fkl_{i,j}(C)$ to be the unique basis element $C'\in \dkl_{j,i}$ that appears as a summand of $\T{j} C$.

\end{definition}
\begin{observation}\label{fbasispermutation}
Theorem \ref{klexchange}  showed that $\fkl_{i,j}(C_w)=C_{\fsn_{i,j}(w)}$. 
\end{observation}

We constructed $\tau$-invariants and the maps $\fkl$ directly using intrinsic properties of the natural $S_n$-action on $\Hn$ and the \kl basis elements.  In fact, as the previous observations hint and the next theorem proves, this construction is completely consistent with the analogous definitions for permutations and for Young tableaux.    
\begin{theorem}[The Robinson--Schensted algorithm for Left Cell Representations]\label{theorem: RS for left cells} Let $\CC$ be a \kl left cell in $S_n$ and $C_w$ a basis element in $B(\klc)$. Then $P(w)$, the left Robinson--Schensted tableau for $w$, is the unique standard tableau on $n$ boxes such that \[\tau_g(P(w))=\tau_g(C_w).\]
\end{theorem}

\begin{proof} 
Let $\Phi: B(\klc) \longrightarrow S_n$ be the map that sends each $C_w\in B(\klc)$ to $w\in S_n$. Consider the following diagram:
\[\begin{picture}(230, 70)(-20,0)
\put(0,3){\vector(1,0){97}}
\put(-15,0){$S_n$}
\put(100,0){$S_n$}
\put(45,10){$\fsn$}

\put(-30,50){$B(\klc)$}
\put(85,50){$B(\klc)$}
\put(8,53){\vector(1,0){75}}
\put(45,60){$\fkl$}

\put(-10,43){\vector(0,-1){32}}
\put(-20,25){$\Phi$}
\put(105,43){\vector(0,-1){32}}
\put(108,25){$\Phi$}

\put(114,45){\vector(2,-1){34}}
\put(135,40){$\tau$}
\put(114,5){\vector(2,1){34}}
\put(135,5){$\tau$}
\put(154,23){subsets of $S_n$}
\end{picture}\]
Observation \ref{taubasispermutation} says that the maps in the triangle commute.  Observation \ref{fbasispermutation}  (or Theorem \ref{klexchange}) says that the maps in the square commute.  Using the definition of generalized $\tau$-invariants and the appropriate commutative diagram, we conclude that $\tau_g(C_w) = \tau_g(w)$ for each $C_w\in \CC$.  Theorem \ref{RSCtheorem} completes the proof. 
\end{proof}

In Theorem \ref{leftcellunique}, Kazhdan-Lusztig and Bj\"{o}rner gave an isomorphism between the left cell representations for any two tableaux of the same shape, an isomorphism
that restricts to a bijection between Kazhdan--Lusztig left cell bases.  Their result implies that all left cell bases for a given shape are equivalent and that the parameterization of these bases by left tableaux is canonical, independent of the underlying permutations. Theorem \ref{theorem: RS for left cells} enhances this result by allowing us to recover the left tableau for each basis element in $B(\klc)$ even if we ``forget" the original permutation. This canonical identification between basis elements and Young tableaux is at the heart of the results in the next sections.

\section{ The Braid Group Action on \texorpdfstring{$\la{sl}_3$}{sl3}-webs}\label{snaction}

In this section, we discuss fundamental properties of $\la{sl}_3$-webs, including the braid group action that webs carry and the basis of {\em reduced webs}.  As we will discuss in Section \ref{webandkl}, this braid group action on the space of webs comes from a Hecke algebra action that deforms the permutation action on tensor factors; in fact, the space of webs $W_{3n}$ discussed here is isomorphic to the $[n,n,n]$ \kl left cell module.  In the next sections, we will recall Khovanov--Kuperberg's bijection between reduced webs and standard tableaux and then show the main result of this paper: Khovanov--Kuperberg's bijection is analogous to the  constructions we gave in earlier sections of this paper and provides a natural analogue of the Robinson--Schensted algorithm for webs. In the last section, we will {\em dis}prove the natural conjecture that the reduced web basis is in fact the \kl basis.  

The $\la{sl}_3$ spider, introduced by Kuperberg~\cite{K} and subsequently studied by many others~\cite{KK, Kim, MR2710589, MOY}, is a diagrammatic, braided monoidal category encoding the representation theory of $\C U_{q}(\la{sl}_3)$.  Representation theoretically, the objects in the spider are tensor products of the two dual 3-dimensional irreducible representations of $\C U_q(\la{sl}_3)$, and the morphisms are intertwining maps. 
Spider categories were first defined for simple Lie types $A_1$, $B_2$, and $G_2$ \cite{K} and then in other cases~\cite{FON, MR2710589, MOY}, but we use only $\la{sl}_3$ in this paper. 

\vspace{0in}

\begin{figure}[ht]
\includegraphics[width=1in]{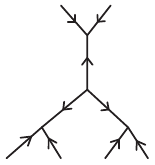}
\caption{A web in $\Hom(--{},{}++++)$}\label{webexample}
\end{figure}

Combinatorially, the objects in the $\la{sl}_3$ spider category are finite strings in the alphabet $\{+,-\}$, including the empty string. The morphisms are $\mathbb{Z}[q, q^{-1}]$-linear combinations of certain graphs called {\em webs}.  While one can write down an explicit formula for the intertwining map corresponding to a web (see for instance \cite{KK, MR2710589}), we will not use this here. 

\begin{definition} Webs are directed trivalent planar graphs with boundary
\begin{itemize}
\item  whose interior vertices are either sources or sinks and 
\item whose boundary vertices are incident to exactly one edge and are located at the top or bottom of a square region.  
\end{itemize}
\end{definition}

The sign $\pm$ of the boundary vertices at the top or bottom of the square region correspond to the domain and codomain respectively of the web. A vertex sign can be determined by the orientation of the edge at that vertex; edges pointing up yield a $+$ while edges pointing down yield a $-$. (Many papers place the domain on the bottom instead.) Figure \ref{webexample} gives an example of a web.

Local relations \ref{circlerel}, \ref{bigonrel}, and \ref{squarerel}  describe equivalences between webs, often called the circle, bigon, and square relations.  {\em Reduced webs} are those with no circles, squares, or bigons.  We follow Khovanov's normalization conventions for Relations \ref{circlerel} -- \ref{equation: smoothing}~\cite{MK}. 
\begin{equation} \label{circlerel}
\raisebox{-11pt}{\includegraphics[width=.4in]{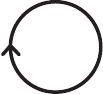}} = [3]_q = q^2 + 1 + q^{-2} 
 \end{equation}
 \begin{equation}\label{bigonrel}
 \raisebox{-20pt}{\includegraphics[height=.6in]{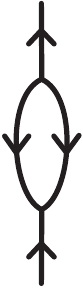}}= [2]_q \raisebox{-17pt}{\includegraphics[height=.5in]{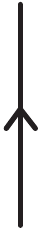}} = (q+q^{-1})  \raisebox{-17pt}{\includegraphics[height=.5in]{BigonR.pdf}}
 \end{equation}
 \begin{equation}\label{squarerel}
 \raisebox{-15pt}{\includegraphics[height=.45in]{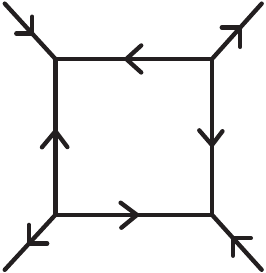}} = \raisebox{-15pt}{\includegraphics[height=.45in]{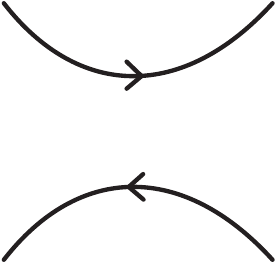}} +  \raisebox{-15pt}{\includegraphics[height=.45in]{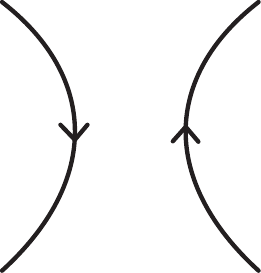}}
\end{equation}
Objects in the $\la{sl}_3$ spider come equipped with a braid group action that commutes with the $\C U_q(\la{sl}_3)$-action.  We often refer to the action of elements of the braid group as {\em braiding}.  Figure \ref{braidskeinrels} shows how braid crossings are locally interpreted in the $\la{sl}_3$ spider. The relations in Figure \ref{braidskeinrels} (and their $U_q(\la{sl}_n)$ counterparts) are often displayed in bracket notation as they are used to obtain knot invariants.   We describe a related Hecke algebra action in Section \ref{webandkl}.  
\begin{figure}[h]\label{braidingmorphisms}
\begin{equation}
\begin{tikzpicture}[baseline=.25cm, scale=.2]
\draw[style= thick, ->] (0,0)--(2,4);
\draw[style=thick,] (2,0) -- (1.25, 1.5);
\draw[style=thick, ->] (.75, 2.5) -- (0,4);
\node at (4, 2) {$= q^2$};
\draw [style=thick, ->] (6,0) to[out=45, in=270] (7,2);
\draw [style=thick] (7,2) to[out=90, in=110] (6,3.75);
\draw [style=thick, ->] (9,0) to[out=135, in=270] (8,2);
\draw [style=thick] (8,2) to[out=90, in=70] (9,3.75);
\node at (11, 2) {$- q^3$};
\draw [style = thick, ->] (13,0) -- (13.5, .5);
\draw [style = thick] (13.5,.5) -- (14, 1);
\draw [style = thick, ->] (15,0) -- (14.5, .5);
\draw [style = thick] (14.5,.5) -- (14, 1);
\draw[style=thick,](14, 1) -- (14, 2);
\draw[style=thick, <-](14, 2) -- (14, 3);
\draw [style = thick, ->] (13,4) -- (13.5, 3.5);
\draw [style = thick] (13.5,3.5) -- (14, 3);
\draw [style = thick, ->] (15,4) -- (14.5, 3.5);
\draw [style = thick] (14.5,3.5) -- (14, 3);

\end{tikzpicture}\label{poscrosseq}
\end{equation}

\begin{equation}\label{equation: smoothing}
\begin{tikzpicture}[baseline=.25cm, scale=.2]
\draw[style= thick, ->] (2,0)--(0,4);
\draw[style=thick,] (0,0) -- (.75, 1.5);
\draw[style=thick, ->] (1.25, 2.5) -- (2,4);
\node at (4.5, 2) {$= q^{-2}$};
\draw [style=thick, ->] (7,0) to[out=45, in=270] (8,2);
\draw [style=thick] (8,2) to[out=90, in=110] (7,3.75);
\draw [style=thick, ->] (10,0) to[out=135, in=270] (9,2);
\draw [style=thick] (9,2) to[out=90, in=70] (10,3.75);
\node at (12, 2) {$- q^{-3}$};
\draw [style = thick, ->] (14,0) -- (14.5, .5);
\draw [style = thick] (14.5,.5) -- (15, 1);
\draw [style = thick, ->] (16,0) -- (15.5, .5);
\draw [style = thick] (15.5,.5) -- (15, 1);
\draw[style=thick,](15, 1) -- (15, 2);
\draw[style=thick, <-](15, 2) -- (15, 3);
\draw [style = thick, ->] (14,4) -- (14.5, 3.5);
\draw [style = thick] (14.5,3.5) -- (15, 3);
\draw [style = thick, ->] (16,4) -- (15.5, 3.5);
\draw [style = thick] (15.5,3.5) -- (15, 3);

\end{tikzpicture}
\end{equation}
\caption{The braiding morphisms in the $\la{sl}_3$ spider}\label{braidskeinrels}
\end{figure}

We extend to complex coefficients in this paper.  Moreover, we take the classical limit, namely let $q=-1$.  This simplifies Relations \ref{circlerel} and \ref{bigonrel}, with $[3]_{q=-1} = 3$ and $[2]_{q=-1} = -2$.  The two braiding relations  in Figure \ref{braidskeinrels} become equal and each term on the right-hand side becomes $+1$.  For example, if $s_i$ is the positive crossing from $i$ to $i+1$ then a small calculation in this classical limit shows $s_i^2=I$.  We use this fact later to construct a natural $S_n$-action on the space of webs.

\begin{definition}
Let $W_{3n} = \Hom(\emptyset,{}+++ \cdots +++)$. Note that $W_{3n}$ is generated as a $\mathbb{Z}[q,q^{-1}]$-module by the webs with lower boundary a string of $3n$ +'s and upper boundary the empty word. Let $B_{3n} \subset \Hom(+++ \cdots +++{}, +++ \cdots +++{})$ be the subset consisting of all braiding morphisms with $3n$ strands. 
\end{definition}

The group $B_{3n}$ is in fact the braid group and acts on $W_{3n}$ by composition. Given $b\in B_{3n}$ and $w\in W_{3n}$ we denote the composition by $b\circ w\in W_{3n}$.  The symmetric group $S_{3n}$ is a quotient of the the braid group $B_{3n}$ by adding the relations $s_i^2=I$ for each $i$.  We observed that in the classical limit, these relations hold, so the $B_{3n}$-action on $W_{3n}$ descends to an $S_{3n}$-representation.  

\begin{remark}
We typically omit orientations in our graphs since they are uniquely determined by the fact that edges point away from the boundary.  For instance, we compute the $S_{3n}$ action using the diagram in Figure \ref{1.1}.
\end{remark}

Petersen--Pylyavskyy--Rhoades proved that the $S_{3n}$--representation $W_{3n}$ with basis consisting of reduced webs  is isomorphic to the irreducible $S_{3n}$-representation for the partition $[n,n,n]$ (with the correspondence in, e.g., Theorem \ref{thm:snrepresentations})~\cite[Lemma 4.2]{PPR}. One can also use Schur--Weyl duality to prove these representations are isomorphic \cite{MR0000255,MR1321638,MR841713} though classical Schur--Weyl duality does not consider the images of specific basis elements.  

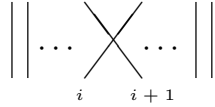
\begin{figure}[ht]
\scalebox{1.2}{\begin{picture}(55, 35) (-30,0)
\put(-23,5){\line(0,1){24}}
\put(-18,5){\line(0,1){24}}
\put(-15,12){$\cdots$}
\put(0,5){\line(3,4){18}}
\put(18,5){\line(-3,4){18}}
\put(18,12){$\cdots$}
\put(35,5){\line(0,1){24}}
\put(40,5){\line(0,1){24}}
\put(-3,-2){\Tiny{$i$}}
\put(14,-2){\Tiny{$i+1$}}
\end{picture}}
\caption{A diagram for $s_i$}\label{1.1}
\end{figure}

Figure \ref{permact} gives an example of a permutation acting on a web, in this case the permutation $\sigma = s_2 s_1\in S_6$.  
The permutation $\sigma$ has two crossings, each of which can be resolved in two different ways per Figure \ref{braidskeinrels}.   Thus the first equation in Figure \ref{permact} shows four web terms. The final lines in the calculation use the bigon and square relations and further simplify.  
\begin{figure}[ht]
\begin{picture}(200,170)(50,40)
\put(15,180){$D_{\sigma} =$}
\put(50,165){\includegraphics[height=.6in]{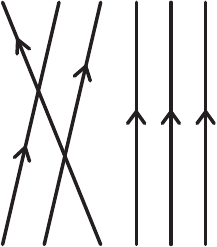}}
\put(160,180){$W =$}
\put(190,165){\includegraphics[height=.5in]{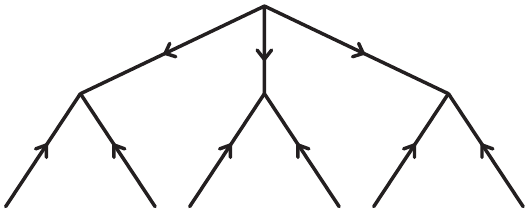}}
\put(50,130){$\sigma\cdot W  =$}
\put(110,110){\includegraphics[height=.6in]{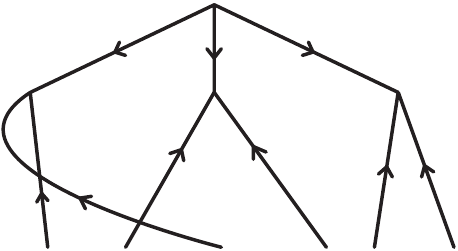}}
\put(200,130){$=$}
\put(-30,55){\ \raisebox{0in}{\includegraphics[height=.5in]{web1}} \raisebox{15pt}{$+$} \includegraphics[height=.5in]{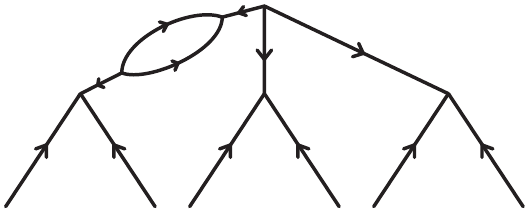} \raisebox{15pt}{$+$} \includegraphics[height=.5in]{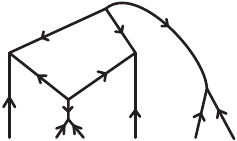} \raisebox{15pt}{$+$} \includegraphics[height=.5in]{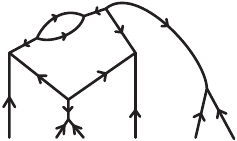}}
\end{picture}
\[=-\left(\raisebox{-15pt}{\includegraphics[height=.5in]{web1}}+\raisebox{-15pt}{\includegraphics[height=.5in]{web6}}\right)
\]
\[=-\left( \raisebox{-.25in}{\includegraphics[height=.5in]{web1}} + \raisebox{-.2in}{\includegraphics[height=.4in]{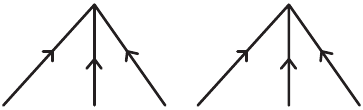}} + \raisebox{-.2in}{\includegraphics[height=.4in]{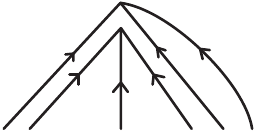}}\right)\]
\caption{The action of the permutation $s_2s_1$ on the web $W$.} \label{permact}
\end{figure}

\begin{remark}
In Kuperberg's work, webs are themselves the basis elements of the invariant subspace of certain ${U}_q(\la{sl}_3)$-representations.  By contrast, in our work the webs are a formal index set for a related basis of an $S_{3n}$-representation.   We will conflate the two in our notation, e.g. referring to specific webs as basis elements.
\end{remark}

\section{Khovanov--Kuperberg's bijection}\label{bijectionsection}
In this paper, we have introduced several different ways to parametrize the basis vectors of a representation: Young tableaux, permutations, and now webs.  We want a direct relationship between webs and tableaux to prove our main results.  Khovanov and Kuperberg introduced a bijection between reduced webs and dominant lattice paths in the weight lattice for $\la{sl}_3$ \cite{KK}. When all base vertices of the web are sources, we may interpret this as a bijection between reduced webs on $3n$ base vertices and $[n,n,n]$ standard Young tableaux \cite{PPR}.  This section describes Khovanov--Kuperberg's bijection in detail.

Khovanov--Kuperberg's map sends each web to a \emph{Yamanouchi word} which is then used to build a standard tableau. 

\begin{definition}
A Yamanouchi word is a string of symbols in the alphabet $\{ +, 0, -\}$ with the property that at any point in the word, the number of $+$'s to the left is greater than or equal to the number of $0$'s to the left which is greater than or equal to the number of $-$'s to the left. A Yamanouchi word is said to be balanced if it has $n$ of each symbol for some $n\in \mathbb{N}$.
\end{definition}

Figure \ref{depthcalculation} has an example of a balanced Yamanouchi word. Yamanouchi words can be defined on other ordered set of symbols. The following theorem connects balanced Yamanouchi words and standard tableaux \cite[p 68]{MR1464693}.

\begin{theorem}[Fulton]
Balanced Yamanouchi words in the alphabet $\{+, -, 0\}$ are in bijection with standard Young tableaux of shape $[n,n,n]$ via the following correspondence.

Let $T$ be a tableau $T$ of shape $[n,n,n]$.  The Yamanouchi word $y_T = y_1y_2\cdots y_{3n-1}y_{3n}$ is a string of symbols in the alphabet $\{+,0,-\}$ where 
\begin{displaymath}
  y_{i} = \left\{
     \begin{array}{ll}
       + & \textup{if $i$ is in the top row of $T$},\\
       0 &  \textup{if $i$ is in the middle row of $T$, and}\\
       - & \textup{if $i$ is in the bottom row of $T$.}\\
            \end{array}
   \right.
\end{displaymath} 
\end{theorem}

Khovanov--Kuperberg's algorithm to build a Yamanouchi word from a reduced web on $3n$ source vertices follows:
\begin{itemize}
\item Draw the source vertices of $W$  on a horizontal line with $W$ in the upper half-plane.
\item 
The vertices and edges in the web divide the upper half-plane into faces, with one infinite face; label the infinite face 0.
\item
Label every other face in the upper half-plane with the minimum number of edges that a path must cross to reach this face from the infinite face.
\item Under each base vertex, write $+$,  $0$, or  $-$ depending on whether the labels on the faces directly above the vertex increase, stay the same or decrease reading from left to right.
\item The string under the horizontal line is a Yamanouchi word. Complete the algorithm by writing down $T$.
\end{itemize}
It is true but not obvious that the algorithm produces a Yamanouchi word corresponding to a standard $[n,n,n]$-tableau \cite{KK,PPR}; Figure \ref{depthcalculation} demonstrates the algorithm.  Section \ref{examples} provides more examples: Figure \ref{examplewebs} contains a list of webs while Figures \ref{exampletableaux1} and \ref{exampletableaux2} show the tableaux that correspond to those webs under Khovanov--Kuperberg's bijection.
\begin{figure}[ht]
\begin{tabular}{|c|c|}
\hline 
 \begin{picture}(144,50)
 \put(0,0)
 {\includegraphics[width=2in]{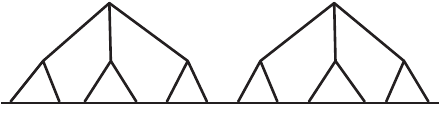}}
 \put(70,24){0}
 \put(97,21){1}
 \put(117,21){1}
 \put(43,21){1}
 \put(25,21){1}
 \put(59,10){1}
 \put(33,10){2}
 \put(10,10){1}
 \put(82.5,10){1}
 \put(107,10){2}
 \put(130,10){1}
 \put(137,0){$-$}
 \put(125,0){0}
 \put(114,0){$-$}
 \put(98,0){$+$}
 \put(88,0){$0$}
 \put(75,0){$+$}
 \put(64,0){$-$}
 \put(53,0){$0$}
 \put(40,0){$-$}
 \put(25,0){$+$}
 \put(17,0){$0$}
 \put(0,0){$+$}
 \end{picture}
 &\raisebox{.1in}{\young(1379,258\eleven,46\ten\twelve)}\\
 
  \hline  
\end{tabular}
\caption{A web, its Yamanouchi word, and its standard tableau.}
\label{depthcalculation}
\end{figure}
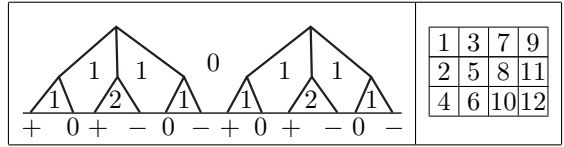

Khovanov and Kuperberg gave an algorithm to compute the reduced web corresponding to a dominant lattice path; sadly, it involves a complicated recursion \cite{KK}.  The third author  developed a simpler method to compute the web corresponding to a standard tableau using intermediate objects called \emph{$M$-diagrams} \cite{T}. 

\begin{definition}
Let $T$ be a standard tableau with three rows.  Construct the $M$-diagram $m_T$ corresponding to $T$ as follows:
\begin{itemize}
\item{Draw a horizontal line with $3n$ dots numbered $1, \ldots, 3n$ from left to right.  (This line forms the lower boundary for the diagram; all arcs lie above it.)}
\item{Start with the smallest number $j$ on the middle row.  Draw a semi-circular arc connecting $j$ to its nearest unoccupied neighbor $i$ to the left that appears in the top row. Continue until all numbers on the middle row have been used.}
\item{Start with the smallest number $k$ on the bottom row.  Draw a semi-circular arc connecting $k$ to its nearest neighbor $j$ to the left that appears in the middle row and does not already have an arc coming to it from the right. Continue until all numbers on the bottom row have been used. }
\end{itemize}
  The arcs $(i,j)$  are the {\em left arcs} and the arcs $(j,k)$ are the {\em right arcs}. 
  \end{definition}
  
The collection of left arcs is nonintersecting by construction, and similarly for the collection of right arcs.  However, left arcs can intersect right arcs. Figure \ref{mdia} shows an example of an $M$-diagram and its corresponding standard tableau.

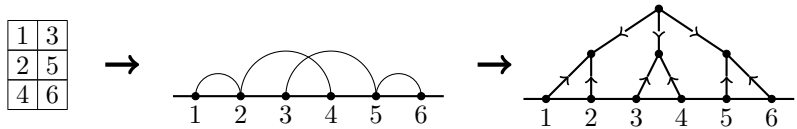
\begin{figure}[ht]
\begin{tikzpicture}[baseline=0cm, scale=.6]
\node at (0,0) {\young(13,25,46)};
\draw[style= ultra thick, ->] (1.5,0)--(2.25,0);
\draw[style=thick] (3,-.7)--(9,-.7);
\draw[radius=.08, fill=black](3.5,-.7)circle;
\draw[radius=.08, fill=black](4.5,-.7)circle;
\draw[radius=.08, fill=black](5.5,-.7)circle;
\draw[radius=.08, fill=black](6.5,-.7)circle;
\draw[radius=.08, fill=black](7.5,-.7)circle;
\draw[radius=.08, fill=black](8.5,-.7)circle;
\draw (4.5,-.7) arc (0:180: .5cm);
\draw (7.5,-.7) arc (0:180: 1cm);
\draw (6.5,-.7) arc (0:180: 1cm);
\draw (8.5,-.7) arc (0:180: .5cm);
\node at (3.5,-1.1) {1};
\node at (4.5,-1.1) {2};
\node at (5.5,-1.1) {3};
\node at (6.5,-1.1) {4};
\node at (7.5,-1.1) {5};
\node at (8.5,-1.1) {6};
\draw[style= ultra thick, ->] (9.75,0)--(10.5,0);
\end{tikzpicture} 
\raisebox{-13pt}{\begin{tikzpicture}[baseline=0cm, scale=0.6]
\draw[style=thick] (2.5,0)--(8.5,0);
\draw[radius=.08, fill=black](3,0)circle;
\draw[radius=.08, fill=black](4,0)circle;
\draw[radius=.08, fill=black](5,0)circle;
\draw[radius=.08, fill=black](6,0)circle;
\draw[radius=.08, fill=black](7,0)circle;
\draw[radius=.08, fill=black](8,0)circle;
\draw[style=thick,->](4, 0) -- (4,.5);
\draw[style=thick](4,.5)--(4,1);
\draw[radius=.08, fill=black](4,1)circle;
\draw[style=thick,->](3,0)--(3.5,.5);
\draw[style=thick](3.5,.5)--(4,1);
\draw[style=thick,->](7, 0) -- (7,.5);
\draw[style=thick](7,.5)--(7,1);
\draw[radius=.08, fill=black](7,1)circle;
\draw[style=thick,->](8,0)--(7.5,.5);
\draw[style=thick](7.5,.5)--(7,1);
\draw[radius=.08, fill=black](5.5,1)circle;
\draw[style=thick,-<](5.5,1)--(5.5,1.5);
\draw[style=thick](5.5,1.5)--(5.5,2);
\draw[radius=.08, fill=black](5.5,2)circle;
\draw[style=thick,->](5,0)--(5.25,.5);
\draw[style=thick](5.25,.5)--(5.5,1);
\draw[style=thick,->](6,0)--(5.75,.5);
\draw[style=thick](5.75,.5)--(5.5,1);
\draw[style=thick,-<](4,1)--(4.75,1.5);
\draw[style=thick](4.75,1.5)--(5.5,2);
\draw[style=thick,-<](7,1)--(6.25,1.5);
\draw[style=thick](6.25,1.5)--(5.5,2);
\node at (3,-.4) {1};
\node at (4,-.4) {2};
\node at (5,-.4) {3};
\node at (6,-.4) {4};
\node at (7,-.4) {5};
\node at (8,-.4) {6};
\end{tikzpicture}}
\caption{The $M$-diagram and web corresponding to a tableau.}\label{mdia}
\end{figure}

A straightforward sequence of resolutions transforms an $M$-diagram $m_T$ into a reduced web $W_T$. 
\begin{itemize}
\item{Replace a small neighborhood the middle boundary vertex of each $m$ with the `Y' shape  in Figure \ref{middle}.}
\item{Orient all arcs away from the boundary so  each boundary vertex is a source.}
\item{Replace each 4-valent intersection of a left arc and a right arc with the pair of trivalent vertices  in Figure \ref{replace4}. (There is a unique way to do this and preserve the orientation of incoming arcs.)}
\end{itemize}

\begin{figure}[h]
\begin{tikzpicture}[baseline=0cm, scale=.5]
\draw[style=thick] (0,0)--(3,0);
\draw[radius=.08, fill=black](1.5,0)circle;
\draw [style=thick](1.5,0) arc (0:100: 1.2cm);
\draw [style=thick] (1.5,0) arc (180:80: 1.2cm);

\draw[style=ultra thick, ->] (3.25,.5) -- (4,.5);

\draw[style=thick] (4.25,0)--(7.25,0);
\draw[radius=.08, fill=black](5.75,0)circle;
\draw[radius=.08, fill=black](5.75,1)circle;
\draw[style=thick,->](5.75,0) -- (5.75,.5);
\draw[style=thick](5.75,.5) -- (5.75,1);
\draw[style=thick, -<] (5.75,1) -- (5.25,1.1);
\draw[style=thick](5.25,1.1) -- (4.75,1.2);
\draw[style=thick, -<] (5.75,1) -- (6.25,1.1);
\draw[style=thick](6.25,1.1) -- (6.75,1.2);
\end{tikzpicture}
\caption{Modifying the middle vertex of an $m$.} \label{middle}
\end{figure}
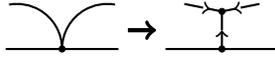

\begin{figure}[ht]
\begin{tikzpicture}[baseline=0cm, scale=0.4]
\draw[style=thick, ->] (-1,-1)--(-.5,-.5);
\draw[style=thick, ->] (-1,1)--(-.5,.5);
\draw[style=thick, ->] (-.5,.5)--(.5,-.5);
\draw[style=thick, ->] (-.5,-.5)--(.5,.5);
\draw[style=thick] (.5,.5)--(1,1);
\draw[style=thick] (.5,-.5)--(1,-1);
\draw[radius=.08, fill=black]circle;

\draw[style=ultra thick, ->] (1.5,0)--(2.5,0);

\draw[xshift=4cm, style=thick, ->] (-1,-1)--(-.5,-.5);
\draw[xshift=4cm, style=thick, ->] (-1,1)--(-.5,.5);
\draw[xshift=4cm, style=thick] (-.5,.5)--(0,0);
\draw[xshift=4cm, style=thick ] (-.5,-.5)--(0,0);
\draw[xshift=4cm, style=thick, -< ] (0,0)--(.5,0);
\draw[xshift=4cm, style=thick ] (.5,0)--(1,0);
\draw[xshift=4.5cm, style=thick,>-] (1,.5)--(1.5,1);
\draw[xshift=4.5cm, style=thick] (.5,0)--(1,.5);
\draw[xshift=4.5cm, style=thick] (.5,0)--(1,-.5);
\draw[xshift=4.5cm, style=thick, >-] (1,-.5)--(1.5,-1);
\draw[xshift=4cm, radius=.08, fill=black]circle;
\draw[xshift=5cm, radius=.08, fill=black]circle;
\end{tikzpicture}
\caption{Replacing a 4-valent vertex with trivalent vertices.}\label{replace4}
\end{figure}
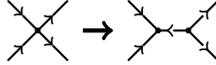

Figure \ref{mdia} shows the web corresponding to an $M$-diagram.  

The third author proved that the map from standard tableaux through $M$-diagram to webs is the inverse of Khovanov--Kuperberg's bijection \cite[Theorem 4.9]{T}.

\section{Khovanov--Kuperberg's bijection as Robinson--Schensted Analogue}\label{bijectionandtau}

The previous sections defined generalized $\tau$-invariants for various combinatorial objects associated to permutations (including Young tableaux and \kl basis elements) and then proved that natural bijections carrying these combinatorial objects to each other also preserve generalized $\tau$-invariants.  This section contains our main result: we define a generalized $\tau$-invariant for webs and prove that Khovanov--Kuperberg's bijection between Young diagrams and irreducible webs also preserves generalized $\tau$-invariants.  In this sense, we show that Khovanov--Kuperberg's bijection gives an analogue of the Robinson--Schensted correspondence.

We begin with a collection of computational lemmas to define generalized $\tau$-invariants for webs and to understand the action of the simple transposition $s_i$ on each irreducible web.  All webs in this section have $3n$ boundary vertices, each of which is a source.

\begin{definition}
Given an irreducible web $W$ the $\tau$-invariant $\tau(W)$ is defined to be the collection of simple reflections $s_i$ for which boundary vertices $i$ and $i+1$ are directly connected to the same internal vertex in $W$.
\end{definition}

Our first result is that this definition of $\tau$-invariants commutes with Khovanov--Kuperberg's bijection, in the sense that if $W_T$ is the web corresponding to the standard tableau $T$ then $\tau(W_T)=\tau(T)$.  

\begin{lemma} \label{tausetlemma}
Let $T$ be a standard tableau of shape $[n,n,n]$ and let $W_T$ be the web obtained from $T$ using Khovanov--Kuperberg's bijection.  Then $\tau(T)= \tau(W_T)$.
\end{lemma}

\begin{proof}
Recall that $\tau(T)$ is the set of all simple transpositions $s_i$ for which $i$ is in a row above $i+1$ in the tableau $T$.  Suppose that $s_i\in \tau(T)$. Since $T$ has three rows, there are three possibilities:
\begin{enumerate}
\item{$i$ is in the top row, and $i+1$ is in the middle row.}
\item{$i$ is in the middle row, and $i+1$ is in the bottom row.}
\item{$i$ is in the top row, and $i+1$ is in the bottom row.}
\end{enumerate}

Consider the $M$-diagram $M_T$ of $T$.  In the first case, the boundary vertices $i$ and $i+1$ must be connected by the left arc of an $M$. Otherwise, two left arcs cross in $M_T$, which cannot happen. Since $i$ and $i+1$ are adjacent and connected by the arc of an $M$, they will be connected to the same internal vertex in $W_T$ as in Figure \ref{middle}.

The second case is similar, with $i$ and $i+1$ now connected by the right arc of an $M$ but still connected to the same internal vertex in $W_T$  as in Figure \ref{middle}.

In the third case, the vertex $i$ is at the far left of an $M$ and $i+1$ is at the far right of an $M$. No boundary vertices lie between $i$ and $i+1$ so the two arcs must cross exactly as in Figure \ref{tauproof1}. Figure \ref{tauproof1} also resolves this crossing, showing that vertices $i$ and $i+1$ connect to the same internal vertex in $W_T$.  

\begin{figure}[ht]
\scalebox{1.5}{\begin{picture}(100,35)(50,0)
\put(0,0){\includegraphics[width=1.25in]{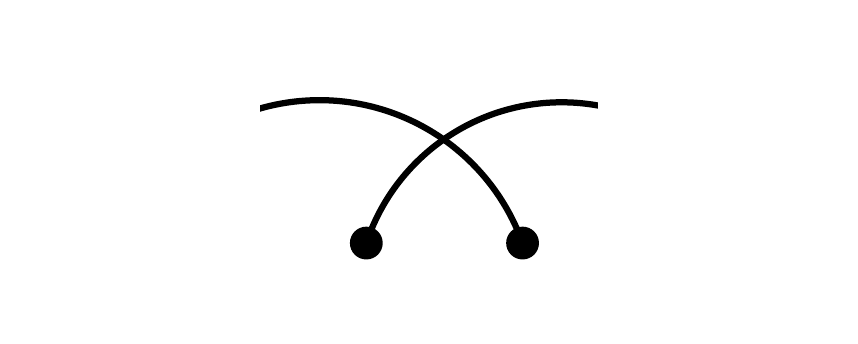}}
\put(85,10){\scalebox{1.5}{$\leadsto$}}
\put(110,0){\includegraphics[width=1.25in]{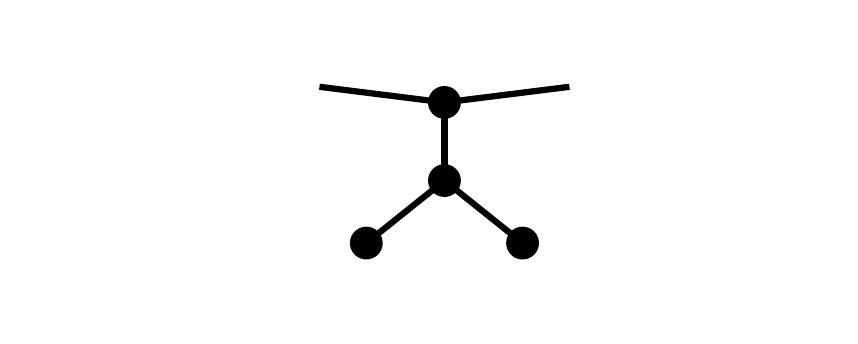}}
\put(35,0){\tiny{$i$}}
\put(50,0){\tiny{$i+1$}}
\put(145,0){\tiny{$i$}}
\put(160,0){\tiny{$i+1$}}
\end{picture}}
\caption{Case (3): Vertices $i$ and $i+1$ in an $M$-diagram and web.}\label{tauproof1}
\end{figure}
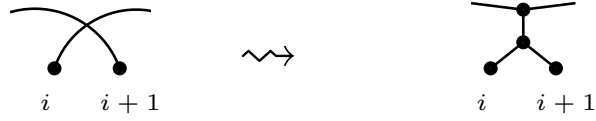

Conversely, suppose $i$ and $i+1$ connect to the same internal vertex in $W_T$.  Consider the possible depths $d_1, d_2$ and $d_3$ of adjacent faces  and the subwords of the Yamanouchi word these depths determine. Depths of faces sharing an edge differ by at most one. Figure \ref{tauproof2} lists the three possibilities for $d_1$ relative to $d_2$ and the consequences for the Yamanouchi word:
\begin{enumerate}
\item{If $d_2=d_1+1$ then $d_3 = d_2$, and $y_iy_{i+1} = +0$.}
\item{If $d_2=d_1$ then $ d_3 = d_1+1$, and $y_iy_{i+1}=+-$.}
\item{If $d_2=d_1-1$ then $d_3 = d_1$, and $y_iy_{i+1} = 0-$.}
\end{enumerate}

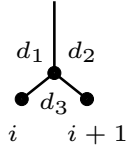
\begin{figure}[ht]
\scalebox{1.5}{\begin{picture}(50, 30)(25,5)
\put(37,21){\tiny{{\bf$d_1$}}}
\put(50,21){\tiny{\bf{$d_2$}}}
\put(0,0){\includegraphics[width=1.25in]{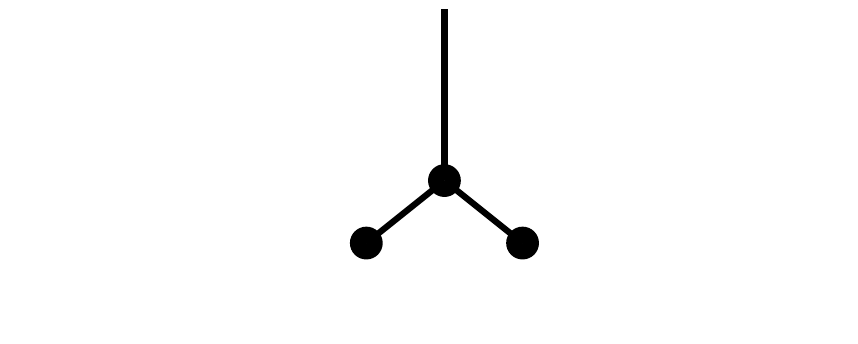}}
\put(35,0){\tiny{$i$}}
\put(50,0){\tiny{$i+1$}}
\put(43,8){\tiny{$d_3$}}

\end{picture}}
\caption{Face depths near vertices $i$ and $i+1$.}\label{tauproof2}
\end{figure}

The letters in a Yamanouchi word tell us the rows of $T$ that contain $i$ and $i+1$:  $+$ indicates the top row, $0$ the middle row, and $-$ the bottom row. In all three cases $i$ is in a higher row than $i+1$ so $s_i \in \tau(T)$.  This proves the claim. 
\end{proof}

Manipulating webs diagrammatically often gives us unreduced webs.  In the next lemma, we study the effects on the $\tau$-invariant and prove that reducing a web cannot remove transpositions from the $\tau$-invariant.

In upcoming proofs, we refer to unbounded faces of webs. A face is considered unbounded if it touches the boundary of the web and thus is not completely enclosed by edges of the web.

\begin{lemma} \label{persistence} 
If $W$ is an unreduced web and $s_i\in \tau(W)$ then $s_i\in \tau(W')$ for every reduced web $W'$ that appears as a summand in the reduction of $W$.
\end{lemma}

\begin{proof}
If $i$ and $i+1$ are connected by an internal vertex $v$ then $v$ only lies on unbounded faces of $W$.   Thus $v$ persists as an internal vertex in all terms of the reduction of $W$. Thus $s_i$ is in the $\tau$-invariant of each of these webs.
\end{proof}

We want to examine the action of simple transpositions on webs more closely in order to build a notion of generalized $\tau$-invariants for webs.  
The next formula identifies $s_i \cdot W$ when $s_i \in \tau(W)$.  The reader should note its similarity to the formula for the action of $T_{s_i}$ on the \kl basis.

\begin{lemma} \label{taurulelemma}
$s_i\cdot W = -W$ if and only if $s_i\in \tau(W)$.
\end{lemma} 

\begin{proof}
First let $s_i\in \tau(W)$. Calculate $s_i\cdot W$ using the diagram for $s_i$ found in Figure \ref{1.1} and then the bigon  rule in Equation \eqref{bigonrel}. Figure \ref{taucalc} has this computation, and shows  $s_i\cdot W = W-2W = -W$. 
\begin{figure}[ht]
$ \raisebox{-.2in}{\includegraphics[width=.2in]{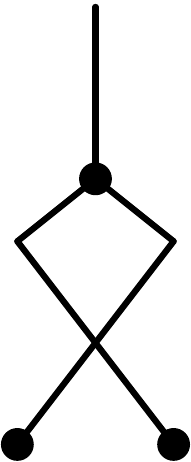}}  = \raisebox{-.2in}{\includegraphics[width=.2in]{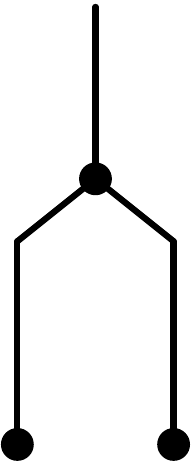}} +\raisebox{-.2in}{ \includegraphics[width=.2in]{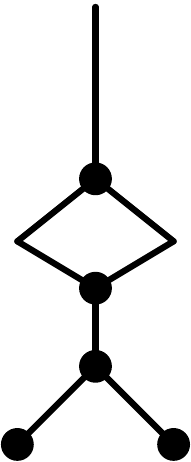}}  = \raisebox{-.2in}{\includegraphics[width=.2in]{web14}} -2\raisebox{-.2in}{ \includegraphics[width=.2in]{web14}} = -\raisebox{-.2in}{ \includegraphics[width=.2in]{web14}}  $
\caption{Calculating $s_i\cdot W$ when $s_i\in \tau(W)$}\label{taucalc}
\end{figure}

Now suppose $s_i\notin\tau(W)$.  Lemma \ref{crossingactioncor} shows $s_i\cdot W =W + W'$ where $W'$ is a possibly unreduced web in which $i$ and $i+1$ are connected by an internal vertex. By Lemma \ref{persistence}, the web $W$ is not a term in the reduction of $W'$ so  $s_i\cdot W \neq -W$.
\end{proof}

We now follow with some more reductions on $s_i \cdot W$.  This is especially useful when $s_i$ is not in $\tau(W)$, in which case we find that $s_i \cdot W$ has at least two summands: $W$ and a web with $s_i$ in its $\tau$-invariant.

\begin{lemma}\label{crossingactioncor}
Let $W$ be a (possibly unreduced) web.  Then $s_i\cdot W = W+W'$ where $W'$ is possibly unreduced and $s_i$ is an element of the $\tau$-invariant of every web in the reduction of $W'$. If $s_i\notin \tau(W)$ then $W$ always appears with coefficient $1$ in $s_i \cdot W$.
\end{lemma}

\begin{proof}
The diagram for $s_i$ in Figure \ref{1.1} has one crossing, so reduce $s_i\cdot W$ by applying the  relations in Figure \ref{braidskeinrels} (with $q=-1$) to the  four-valent vertex in $s_i \cdot W$, as in Figure \ref{figure: figure A}.  This gives a sum of two webs depending on how we smooth the four-valent vertex, also shown in Figure \ref{figure: figure A}. Smoothing into two strands recovers the original web $W$. The other smoothing $W'$ introduces two additional trivalent vertices, one of which joins boundary vertices $i$ and $i+1$. Lemma \ref{persistence} completes the proof.
\end{proof}

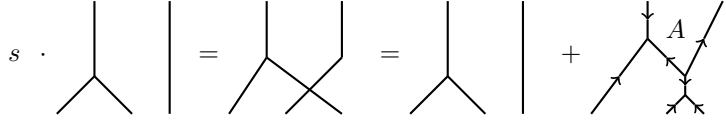
\begin{figure}[h]
\raisebox{20pt}{$s \hspace{.1in} \cdot$} \begin{tikzpicture} [baseline=0cm, scale=0.5]
\draw[style=thick] (-1,0)--(0,1);
\draw[style=thick](0,1)--(1,0);
\draw[style=thick](0,1)--(0,3);
\draw[style=thick] (2,0)--(2,3);
\end{tikzpicture} \hspace{.1in}\raisebox{20pt}{$=$} 
\begin{tikzpicture} [baseline=0cm, scale=0.5]
\draw[style=thick] (-1,0)--(0,1.5);
\draw[style=thick](0,1.5)--(2,0);
\draw[style=thick](0,1.5)--(0,3);
\draw[style=thick] (2,1.5)--(2,3);
\draw[style=thick](2,1.5)--(.5,0);
\end{tikzpicture} \hspace{.1in} \raisebox{20pt}{$=$}
\begin{tikzpicture} [baseline=0cm, scale=0.5]
\draw[style=thick] (-1,0)--(0,1);
\draw[style=thick](0,1)--(1,0);
\draw[style=thick](0,1)--(0,3);
\draw[style=thick] (2,0)--(2,3);
\end{tikzpicture} \hspace{.1in} \raisebox{20pt}{$+$}
\begin{tikzpicture} [baseline=0cm, scale=0.5]
\draw[style=thick, ->] (-1,0)--(-.25,1);
\draw[style=thick](-.25,1)--(.5,2);
\draw[style=thick, ->](.5,3)--(.5,2.5);
\draw[style=thick](.5,2.5)--(.5,2);
\draw[style=thick, ->](1.5,1)--(1,1.5);
\draw[style=thick](1,1.5)--(.5,2);
\draw[style=thick, ->](1.5,1)--(1.5,.75);
\draw[style=thick](1.5,.75)--(1.5,.5);
\draw[style=thick, ->](1,0)--(1.25,.25);
\draw[style=thick](1.25,.25)--(1.5,.5);
\draw[style=thick, ->](2,0)--(1.75,.25);
\draw[style=thick](1.75,.25)--(1.5,.5);
\draw[style=thick, ->](1.5,1)--(2,2);
\draw[style=thick](2,2)--(2.5,3);
\node at (1.25,2.25) {$A$};
\end{tikzpicture}
\caption{Local schematic for how the simple reflection $s$ acts on $W$}\label{figure: figure A}
\end{figure}

As before, we define  sets $\dw_{i,j}$.  

\begin{definition}
Let $s_i$ and $s_j$ be adjacent simple transpositions.  The set $\dw_{i,j}$ consists of all reduced webs $W$ for which $s_i \in \tau(W)$ and $s_j \not \in \tau(W)$.
\end{definition}

The key to defining generalized $\tau$-invariants for webs is an appropriate map $\fw_{i,j}$ which we are not yet in a position to define.  What we still need is a deeper analysis of $s_i \cdot W$ than found in Lemma \ref{crossingactioncor}.  The next lemma does this analysis.

\begin{lemma}\label{webtauinvtlemma}
Let $W \in \dw_{i,j}$. Then $s_j \cdot W = W + W' + O$ where $W'$ is a reduced web in $\dw_{j,i}$ and $O$ is a $\mathbb{Z}$-linear combination of reduced webs whose $\tau$-invariants each contain $s_i$ and $s_j$.  Figure \ref{figure: figure B} shows the web $s_j \cdot W$ and its reduction locally near the affected boundary vertices.
\end{lemma}

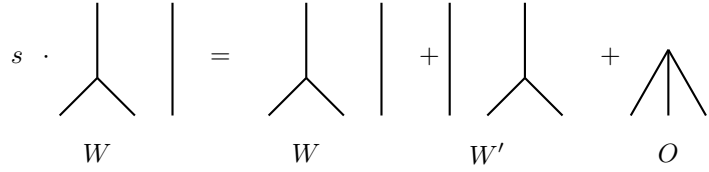
\begin{figure}[h]
\raisebox{20pt}{$s \hspace{.1in} \cdot$} \begin{tikzpicture} [baseline=0cm, scale=0.5]
\draw[style=thick] (-1,0)--(0,1);
\draw[style=thick](0,1)--(1,0);
\draw[style=thick](0,1)--(0,3);
\draw[style=thick] (2,0)--(2,3);
\node at (0,-1) {$W$};
\end{tikzpicture} \hspace{.1in} \raisebox{20pt}{$ = $} \hspace{.1in}
\begin{tikzpicture} [baseline=0cm, scale=0.5]
\draw[style=thick] (-1,0)--(0,1);
\draw[style=thick](0,1)--(1,0);
\draw[style=thick](0,1)--(0,3);
\draw[style=thick] (2,0)--(2,3);
\node at (0,-1) {$W$};
\end{tikzpicture} \hspace{.1in} \raisebox{20pt}{$+$}
\begin{tikzpicture} [baseline=0cm, scale=0.5]
\draw[style=thick] (0,0)--(1,1);
\draw[style=thick](1,1)--(2,0);
\draw[style=thick](1,1)--(1,3);
\draw[style=thick] (-1,0)--(-1,3);
\node at (0,-1) {$W'$};
\end{tikzpicture} \hspace{.1in} \raisebox{20pt}{$+$}
\begin{tikzpicture} [baseline=0cm, scale=0.5]
\draw[style=thick](-1,0)--(0,1.75);
\draw[style=thick](0,0)--(0,1.75);
\draw[style=thick](1,0)--(0,1.75);
\node at (0,-1) {$O$};
\end{tikzpicture}
\caption{Local schematic for reducing the web $s \cdot W$}\label{figure: figure B}
\end{figure}

\begin{proof}
Assume without loss of generality that $j=i+1$.  Compute $s_{i+1}\cdot W$ using the diagram for $s_{i+1}$  in Figure \ref{1.1} followed by the relation in Figure \ref{braidskeinrels} (with $q=-1$) to the new four-valent vertex in $s_{i+1} \cdot W$.  Figure \ref{figure: figure A} shows this calculation.  If the second web in Figure \ref{figure: figure A} is reduced then it is $W'$ by Lemma \ref{persistence}, and $O$ is a sum over the empty set.  

Suppose the second web is not reduced.  Then the problem is the face labeled $A$, since every other modified face is unbounded.  If the face $A$ were a bigon then $W$ has a tripod on vertices $i,i+1,i+2$ (namely a connected component like those in Figure \ref{counterexample2}).  Thus an internal vertex in $W$ joins  boundary vertices $i+1$ and $i+2$, which means $s_{i+1} \in \tau(W)$ by Lemma \ref{persistence}.  This contradicts the hypothesis on $W$.

This means the face $A$ must be a square.  The square relation in Equation \eqref{squarerel} produces two summands shown in Figure \ref{figure: figure B}.  Denote the second term in Figure \ref{figure: figure B} by $W'$.  It is reduced because every bounded face in $W'$ is also in $W$.  Denote the third term in Figure \ref{figure: figure B} by $O$.  It may not be reduced, but since $\tau(O)$ contains both $s_i$ and $s_{i+1}$ then so do all reduced webs obtained from $O$, by Lemma \ref{persistence}.  This proves the claim.
\end{proof}

The next corollary restates Figure \ref{figure: figure B} in the language of Yamanouchi words.

\begin{corollary}\label{Yamanouchi word corollary}
Let $W$ and $W'$ be as defined in Lemma \ref{webtauinvtlemma}. Then, the Yamanouchi words of $W$ and $W'$ agree except in positions $\{i,i+1,j,j+1\}$.
\end{corollary}

As with $\fyt$, $\fsn$, and $\fkl$, we use this result to define $\fw_{i,j}$.  We  then define generalized $\tau$-invariants for webs following Definition \ref{gentaudef} and  Remark \ref{remark: general gentaudef}.

\begin{definition}
Define the function $\fw_{i,j}: \dw_{i,j} \rightarrow \dw_{j,i}$ by the rule that for each $W \in \dw_{i,j}$ the image $\fw_{i,j}(W)$ is the reduced web $W' \in \dw_{j,i}$ from Lemma \ref{webtauinvtlemma}.
\end{definition}

Figure \ref{examplewebs} shows many examples of webs and their images under $\fw$.

Lemma \ref{tausetlemma} showed that the $\tau$-invariant commutes with Khovanov--Kuperberg's bijection.  We now show that the map $\fw$ commutes with Khovanov--Kuperberg's bijection as well, in the following sense.  

\begin{lemma}\label{lemma: kupbij commutes with f}
Suppose that $W_T$ is the reduced web corresponding to the standard tableau $T$ under Khovanov--Kuperberg's bijection.  If $W_T \in \dw_{i,j}$ then $\fw_{i,j}(W_T) = W_{\fyt_{i,j}(T)}$.
\end{lemma}

\begin{proof}
We first confirm that $\fyt_{i,j}$ is well-defined on $T$.  Indeed $\tau(W_T) = \tau(T)$ by Lemma \ref{tausetlemma}, so $W_T \in \dw_{i,j}$ if and only if $T \in \dyt_{i,j}$. 
We will show that either $T' = s_i \cdot T$ or $T' = s_j \cdot T$.  The result then follows from the construction of $\fyt$ in Lemma \ref{fytlemma}.

Corollary \ref{Yamanouchi word corollary} showed that $W_T$ and $W'$ have the same Yamanouchi words except in positions $\{i,i+1,j,j+1\}$.  (This is the content of Figure \ref{figure: figure B} when $j=i+1$.) 

Our argument uses Yamanouchi subwords, like the proof of Lemma \ref{tausetlemma}. Figure \ref{Yamanouchicases} lists the eight options for the subword $y_iy_{i+1}y_{i+2}$ in the Yamanouchi word for $T$. In each case, the rows containing $i, i+1, i+2$  locally determine the $M$-diagram for $T$.  (In some cases, there are two possibilities.) Figure \ref{Yamanouchicases} also shows the local $M$-diagrams.  We labeled the number of arcs above each region in the $M$-diagrams in Figure \ref{Yamanouchicases}; Tymoczko showed that this is the depth of the corresponding region in the resolved web \cite[Lemma 4.5]{T}.  

\begin{figure}[ht]
\begin{tabular}{|r|c|c|c|c|}
\hline&&&\multicolumn{2}{|c|}{}\\
& {\bf $T$} &{\bf $\fyt_{i,j}(T)$} & \multicolumn{2}{|c|}{Possible $M$-Diagram(s) for $T$}\\
&&&\multicolumn{2}{|c|}{}\\
\hline &&&&\\

  {\bf Case 1} & $+0+$& $++0$ &
  \begin{picture}(100,20)(0,0) 
  \put(5,-10){\scalebox{1.25}{\includegraphics[width=1in]{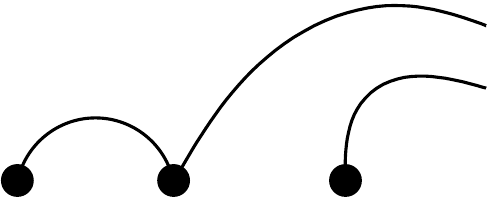}}} 
  \put(13,-5) {\tiny{$d+1$}} 
  \put(30,10) {\Small{$d$}} 
  \put(45,-5) {\tiny{$d+1$}} 
  \put(75,-5) {\tiny{$d+2$}} 
  \end{picture} & 
  \begin{picture}(110,20)(0,0) 
  \put(5,-10){\scalebox{1.25}{\includegraphics[width=1in]{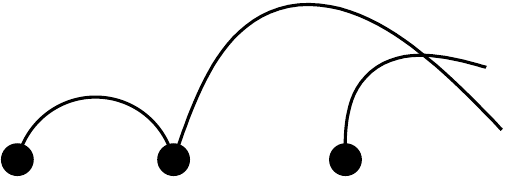}}}
   \put(13,-10) {\tiny{$d+1$}} 
   \put(20, -3) {\color{red}{\tiny{$B$}}}
     \put(30,10) {\Small{$d$}} 
  \put(45,-10) {\tiny{$d+1$}} 
    \put(50, 0) {\color{red}{\tiny{$C$}}}
  \put(75,-10) {\tiny{$d+2$}} 
  \put(95,0){\tiny{$d+1$}}
    \put(95, 5) {\color{red}{\tiny{$F$}}}
  \end{picture} \\
  &&&&\\
\hline  &&& \multicolumn{2}{|c|}{}\\
{\bf Case 2} & $+00$& $0+0$ & \multicolumn{2}{|c|}{ 
\begin{picture}(150,30)(0,0)
\put(30,-10){\scalebox{1.15}{\includegraphics[width=1in]{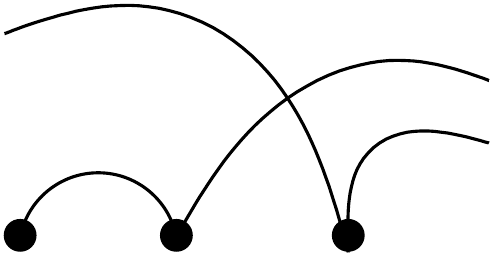}}}
\put(38,-10){\tiny{$d+2$}}
  \put(42, -3) {\color{red}{\tiny{$B$}}}
\put(65,-10){\tiny{$d+2$}}
  \put(70, 0) {\color{red}{\tiny{$C$}}}
\put(95,-10){\tiny{$d+2$}}
  \put(105, 14) {\color{red}{\tiny{$F$}}}
\put(47,10){\tiny{$d+1$}}
\put(85,14){\tiny{$d+1$}}
\put(75,27){\tiny{$d$}}
\end{picture}} \\
  &&& \multicolumn{2}{|c|}{}\\
  \hline  &&& \multicolumn{2}{|c|}{}\\
 {\bf Case 3} & $+-+$& $++-$ &  \multicolumn{2}{|c|}{ 
 \begin{picture}(150,30)(0,0)
\put(30,-10){\scalebox{1.15}{\includegraphics[width=1in]{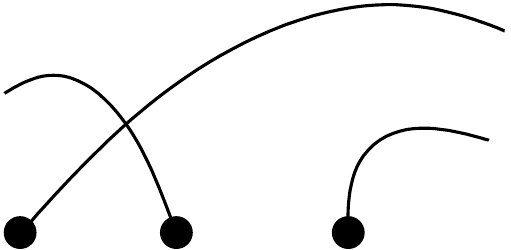}}}
\put(38,-7){\tiny{$d+2$}}
\put(65,7){\tiny{$d+1$}}
\put(95,-7){\tiny{$d+2$}}
\put(25,7){\tiny{$d+1$}}
\put(65,27){\tiny{$d$}}
\end{picture}} 
 \\
&&& \multicolumn{2}{|c|}{}\\
  \hline  &&& \multicolumn{2}{|c|}{}\\
  {\bf Case 4} & $+-0$& $-+0$ & \multicolumn{2}{|c|}{ 
  
  \begin{picture}(150,30)(0,0)
\put(30,-10){\scalebox{1.15}{\includegraphics[width=1in]{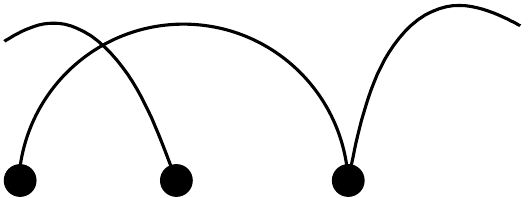}}}
\put(38,-7){\tiny{$d+2$}}
\put(60,-3){\tiny{$d+1$}}
\put(95,0){\tiny{$d+1$}}
\put(15,6){\tiny{$d+1$}}
\put(65,27){\tiny{$d$}}
  \put(50, 25) {\color{red}{\tiny{$F$}}}
    \put(42, 0) {\color{red}{\tiny{$B$}}}
      \put(65, 5) {\color{red}{\tiny{$C$}}}
\end{picture}} \\
&&& \multicolumn{2}{|c|}{}\\
  \hline  &&& \multicolumn{2}{|c|}{}\\
 {\bf Case 5} & $+--$ & $-+-$ &\multicolumn{2}{|c|}{
 
   \begin{picture}(150,30)(0,0)
\put(30,-10){\scalebox{1.15}{\includegraphics[width=1in]{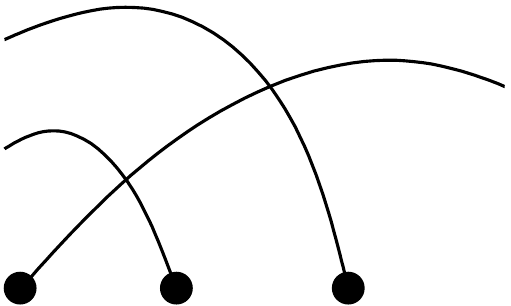}}}
\put(38,-7){\tiny{$d+3$}}
\put(60,-3){\tiny{$d+2$}}
\put(90,7){\tiny{$d+1$}}
\put(15,6){\tiny{$d+2$}}
\put(45,22){\tiny{$d+1$}}
\put(75,35) {\tiny{$d$}}
  \put(40, 30) {\color{red}{\tiny{$F$}}}
    \put(44, 0) {\color{red}{\tiny{$B$}}}
      \put(62, 5) {\color{red}{\tiny{$C$}}}
\end{picture}} \\
&&& \multicolumn{2}{|c|}{}\\
  \hline  &&& \multicolumn{2}{|c|}{}\\
 {\bf Case 6} & $0-+$ & $0+-$ &\multicolumn{2}{|c|}{ 
 \begin{picture}(150,20)(0,0)
\put(25,-10){\scalebox{1.25}{\includegraphics[width=1.2in]{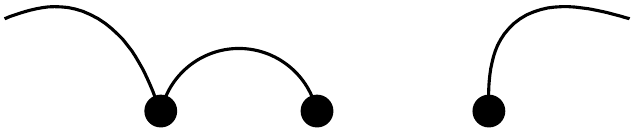}}}
\put(20,-3){\tiny{$d+1$}}
\put(58,-7){\tiny{$d+1$}}
\put(115,-3){\tiny{$d+1$}}
\put(50,8){\Small{$d$}}
\end{picture}} \\
&&& \multicolumn{2}{|c|}{}\\
  \hline  &&& \multicolumn{2}{|c|}{}\\
   {\bf Case 7} & $0-0$ & $00-$ &\multicolumn{2}{|c|}{
    \begin{picture}(150,30)(0,0)
\put(25,-10){\scalebox{1.25}{\includegraphics[width=1.2in]{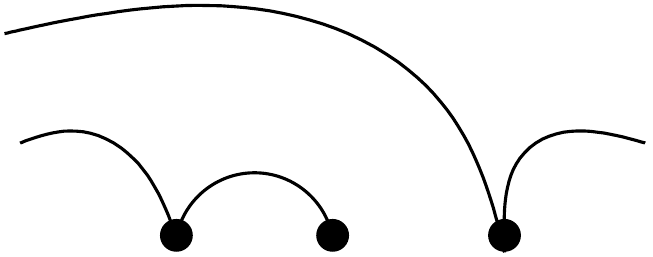}}}
\put(27,-3){\tiny{$d+2$}}
\put(58,-7){\tiny{$d+2$}}
\put(115,-3){\tiny{$d+1$}}
\put(52,10){\Small{$d+1$}}
\put(110,20) {\Small{$d$}}
\end{picture}} \\
&&& \multicolumn{2}{|c|}{}\\
  \hline  &&&&\\
{\bf Case 8} & $0--$ & $-0-$ &
 \begin{picture}(110,30)(0,0) 
  \put(5,-10){\scalebox{1.25}{\includegraphics[width=1in]{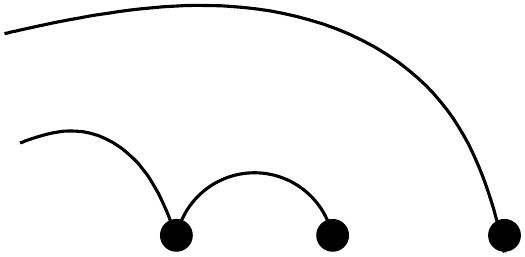}}} 
  \put(13,-5) {\tiny{$d+2$}} 
  \put(40,10) {\Small{$d+1$}} 
  \put(40,-7) {\tiny{$d+2$}} 
  \put(90,12) {\Small{$d$}} 
  \end{picture} & 
  \begin{picture}(110,20)(-10,0) 
  \put(5,-10){\scalebox{1.25}{\includegraphics[width=1in]{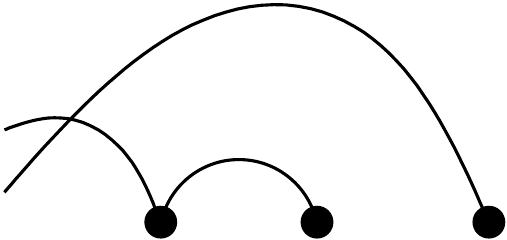}}}
   \put(11,-7) {\tiny{$d+2$}} 
     \put(32,10) {\Small{$d+1$}} 
      \put(65, 8) {\color{red}{\tiny{$C$}}}
  \put(40,-9) {\tiny{$d+2$}} 
    \put(47, -3) {\color{red}{\tiny{$B$}}}
  \put(10,20){\Small{$d$}}
    \put(4, 4) {\color{red}{\tiny{$F$}}}
    \put(-12,-1){\tiny{$d+1$}}
  \end{picture}\\
&&&&\\
\hline \end{tabular}
\caption{Possible Yamanouchi subwords for $T$ and $\fyt_{i,j}(T)$}\label{Yamanouchicases}
\end{figure}

In each case, call the resolved $M$-diagram $W$ and consider $s_i \cdot W$ as in Figure \ref{figure: figure A}.  Either the second web in Figure \ref{figure: figure A} is reduced already, or it contains a square.  The latter occurs when $W$ has exactly three internal vertices between $i+1$ and $i+2$, which happens in cases 1b, 2, 4, 5, and 8b.  The former happens in all other cases: the face containing $i+1$ and $i+2$ has either two first arcs or two second arcs as sides, and only first arcs can cross second arcs. Figure \ref{newfig} illustrates these cases.  After resolving, the face with $i+1$ and $i+2$ has at least four vertices if it is bounded away from $i+1$. 

When there is no square, only faces $B$ and $D$ in Figure \ref{newfig}(a) have new neighbors, so only faces $B$ and $D$ can change depth.  This happens exactly when the depths of $B$ and $D$ differ by two, in which case the larger depth drops by one.  The depth of the new  face $C'$ is $\min \{\textup{depth($B$), depth($D$)}\} + 1$.  We obtain:

\begin{center}\begin{tabular}{l|c|c}
Case & Depth of $A$, $B$, $C$, $D$ & Depth of $A$, $B$, $C'$, $D$ \\
\cline{1-3} 1a & d, d+1, d+1, d+2 & d, d+1, d+2, d+2 \\
3 & d+1, d+2, d+1, d+2 & d+1, d+2, d+3, d+2 \\
6 & d+1, d+1, d, d+1 & d+1, d+1, d+2, d+1 \\
7 & d+2, d+2, d+1, d+1 & d+2, d+2, d+2, d+1 \\
8a & d+2, d+2, d+1, d & d+2, d+1, d+1, d
\end{tabular}\end{center}

Similarly, Figure \ref{newfig}(b) labels the relevant faces in $W$ and $W'$ when there is a square.  (Faces $B$, $C$, and $F$ are shown on the $M$-diagrams in Figure \ref{Yamanouchicases} as well.)  Faces $A$ and $D$ have the same depth in $W$ as in $W'$, since no minimal-length path from $A$ or $D$ goes through $B$ and any path through $C$ can go through $F$ instead.  The face marked $F$ in $W'$ has the same depth as the face marked $F$ in $W$, since no minimal-length path from $F$ goes through $B$ or $C$ and since depth($F$) is at most depth($C$).  Again, the depth of the new  face $C'$ is $\min \{\textup{depth($F$), depth($D$)}\} + 1$. We obtain the following:

\begin{center}\begin{tabular}{l|c|c}
Case & Depth of $A$, $B$, $C$, $D$ & Depth of $A$, $F$, $C'$, $D$ \\
\cline{1-3} 1b & d, d+1, d+1, d+2 & d, d+1, d+2, d+2 \\
2 & d+1, d+2, d+2, d+2 & d+1, d+1, d+2, d+2 \\
4 & d+1, d+2, d+1, d+1 & d+1, d, d+1, d+1 \\
5 & d+2, d+3, d+2, d+1 & d+2, d+1, d+2, d+1 \\
8b & d+2, d+2, d+1, d & d+2, d+1, d+1, d
\end{tabular}
\end{center}

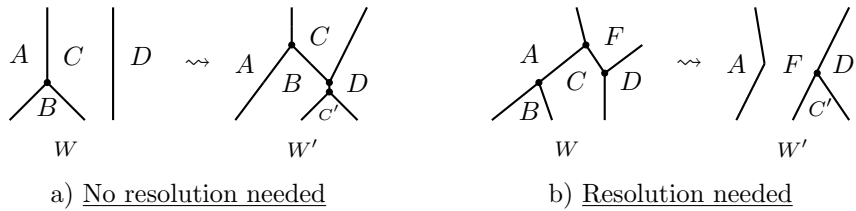
\begin{figure}[h]
\begin{tikzpicture} [baseline=0cm, scale=0.5]
\node at (-.75,1.75) {$A$};
\node at (0,.35) {$B$};
\node at (.7,1.7) {$C$};
\node at (2.5, 1.7) {$D$};
\draw[style=thick] (-1,0)--(0,1);
\draw[radius=.08, fill=black](0,1)circle;
\draw[style=thick](0,1)--(1,0);
\draw[style=thick](0,1)--(0,3);
\draw[style=thick] (1.75,0)--(1.75,3);
\node at (4,1.5){$\leadsto$}; 
\draw[style=thick] (5,0)--(5.75,1);
\draw[style=thick](5.75,1)--(6.5,2);
\draw[radius=.08, fill=black](6.5,2)circle;
\draw[radius=.08, fill=black](7.5,1)circle;
\draw[radius=.08, fill=black](7.5,.75)circle;
\draw[style=thick](6.5,3)--(6.5,2.5);
\draw[style=thick](6.5,2.5)--(6.5,2);
\draw[style=thick](7.5,1)--(7,1.5);
\draw[style=thick](7,1.5)--(6.5,2);
\draw[style=thick](7.5,1)--(7.5,.75);
\draw[style=thick](6.75,0)--(7.5,.75);
\draw[style=thick](8.25,0)--(7.5,.75);
\draw[style=thick](7.5,1)--(8,2);
\draw[style=thick](8,2)--(8.5,3);
\node at (5.25, 1.5) {$A$};
\node at (6.5,1) {$B$};
\node at (7.25,2.25) {$C$};
\node at (7.5, 0.2) {\tiny{$C'$}};
\node at (8.25, 1) {$D$};
\node at (.5, -.75) {\Small{$W$}};
\node at (6.8, -.75) {\Small{$W'$}};
\node at (3.75, -2) {a) \underline{No resolution needed}};
\end{tikzpicture}
\hspace{.5in}
\begin{tikzpicture} [baseline=0cm, scale=0.5]
\draw[style=thick](0,0)--(2.5,2);
\draw[style=thick](2.5,2)--(2.25,3);
\draw[radius=.08, fill=black](2.5,2)circle;
\draw[radius=.08, fill=black](1.25,1)circle;
\draw[style=thick](1.25,1)--(1.6,0);
\draw[style=thick](2.5,2)--(3,1.25);
\draw[radius=.08, fill=black](3,1.25)circle;
\draw[style=thick](3,1.25)--(3,0);
\draw[style=thick](3,1.25)--(4,2);
\node at (1, 1.75) {$A$};
\node at (1, 0.25) {$B$};
\node at (2.25, 1.05) {$C$};
\node at (3.7, 1.05) {$D$};
\node at (3.25,2.25) {$F$};
\node at (5.25,1.5) {$\leadsto$};
\draw[style=thick] (6.5,0)--(7.25,1.5);
\draw[style=thick](7.25,1.5)--(7,3);
\draw[style=thick](8,0)--(9.5,3);
\draw[radius=.08, fill=black](8.65,1.25)circle;
\draw[style=thick](8.65,1.25)--(9.5,0);
\node at (6.5,1.5) {$A$};
\node at (8,1.5) {$F$};
\node at (8.75,.35) {\Small{$C'$}};
\node at (9.35,1.5) {$D$};
\node at (2, -.75) {\Small{$W$}};
\node at (8, -.75) {\Small{$W'$}};
\node at (4.75, -2) {b) \underline{Resolution needed}};
\end{tikzpicture}
\caption{Applying a simple transposition to $W$} \label{newfig}
\end{figure}

Figure \ref{Yamanouchicases} gives the Yamanouchi words for $\fyt_{i,j}(T)$ in each of these cases.  In all cases $T' = s_i \cdot T$ or $T' = s_j \cdot T'$ which proves the claim.
\end{proof}

As in previous sections, we will now show that the generalized $\tau$-invariant we constructed for webs is completely consistent with the generalized $\tau$-invariants defined for Young tableaux.  More precisely, our final and main theorem shows that Khovanov--Kuperberg's bijection commutes with the generalized $\tau$-invariant, in the sense that $\tau_g(T)=\tau_g(W_T)$.  This means that Khovanov--Kuperberg's bijection is  a natural analogue of the Robinson--Schensted correspondence for webs.  
Thus tableaux, permutations, \kl cells, and webs are all naturally related in a way that preserves deep structures associated to them---though they do not correspond to the same bases for representations of $S_n$.

\begin{theorem}[The Robinson--Schensted Correspondence for Webs]\label{maintheorem}
Khovanov--Kuperberg's bijection carries a reduced web $W$ on $3n$ source vertices to the unique $[n,n,n]$ standard tableau $T$ satisfying $\tau_g(T)=\tau_g(W)$.
\end{theorem}

\begin{proof}Let $\varphi: W_{3n} \rightarrow \{[n,n,n] \textup{ standard tableaux}\}$ denote Khovanov--Kuperberg's bijection.  Consider the following diagram:
\[\begin{picture}(230, 85)(-20,-15)
\put(7,3){\vector(1,0){83}}
\put(-30,-3){\small $[n,n,n]$}
\put(-50,-13){\small standard tableaux}
\put(95,-3){\small $[n,n,n]$}
\put(80,-13){\small standard tableaux}
\put(45,10){$\fyt$}

\put(-15,50){$W_{3n}$}
\put(100,50){$W_{3n}$}
\put(3,53){\vector(1,0){90}}
\put(45,60){$\fw$}

\put(-10,43){\vector(0,-1){32}}
\put(-20,25){$\varphi$}
\put(105,43){\vector(0,-1){32}}
\put(108,25){$\varphi$}

\put(114,48){\vector(2,-1){40}}
\put(135,40){$\tau$}
\put(120,5){\vector(2,1){34}}
\put(135,5){$\tau$}
\put(157,22){subsets of $S_n$}
\end{picture}\]
Lemma \ref{tausetlemma} says that the maps in the triangle commute.  Lemma \ref{lemma: kupbij commutes with f} says that the maps in the square commute.  Using the definition of generalized $\tau$-invariants and the appropriate commutative diagram, we conclude that $\tau_g(T) = \tau_g(W)$ for each web $W \in W_{3n}$ and tableau $T = \varphi(W)$.  Theorem \ref{gentautabthm} says that the generalized $\tau$-invariant determines $T$ uniquely, completing the proof. 
\end{proof}

\section{Examples}\label{examples}
This section contains examples of webs and their associated tableaux, together with their $\tau$-invariants and images under $f_{i,j}$. We collect this information to demonstrate ``by hand" that Khovanov--Kuperberg's bijection associates each web to the unique tableau with the same generalized $\tau$-invariant.  At the same time, the reader will observe that several webs or tableaux can share the same $\tau$-invariant.  Figure \ref{examplewebs} shows all the webs, while Figures \ref{exampletableaux1} and \ref{exampletableaux2} show the tableaux.

\subsection{The unique web with a given generalized \texorpdfstring{$\tau$}{tau}-invariant} Consider the tableau $T_1$. After considering all $9$-vertex webs, we find exactly two reduced webs whose $\tau$-invariants agree with that of $T_1$, namely $W_1$ and $W_2$. Now compute the $\tau$-invariants of $\fyt_{1,2}(T_1)$, $\fw_{1,2}(W_1)$ and $\fw_{1,2}(W_2)$. This shows that $W_2$ and $T_1$ have different generalized $\tau$-invariants, so $T_1$ corresponds to $W_1$.  As a final check, Khovanov--Kuperberg's bijection associates $T_1$ with $W_1$.

\subsection{The unique tableau with a given generalized \texorpdfstring{$\tau$}{tau}-invariant}  Now consider the web  $W_3$. After considering all $[4,4,4]$ standard  tableaux, exactly four have the same $\tau$-invariant as $W_3$, namely $T_2$, $T_3$, $T_4$ and $T_5$. Compute the $\tau$-invariant of $\fw_{7,6}(W_3)$ and $\fyt_{7,6}$ of each tableau. At this stage, only $T_2$ and $T_4$ could have the same generalized $\tau$-invariant as $W_3$. We then calculate the $\tau$-invariant of  $\fw_{1,2}(\fw_{7,6}(W_3))$, $\fyt_{1,2}(\fyt_{7,6}(T_2))$ and $\fyt_{1,2}(\fyt_{7,6}(T_4))$. This confirms that $T_4$ and $W_3$ have different generalized $\tau$-invariants, so $W_3$  corresponds to $T_2$.  As before, this is consistent with Khovanov--Kuperberg's bijection.

\begin{figure}[ht]
\begin{tabular}{|c|c|c|}
\hline&&\\ 
& {Web} &{\large $\tau$} \\
&&\\
\hline &&\\

 {\large $W_1$} &\raisebox{-15pt}{\includegraphics[width=1in]{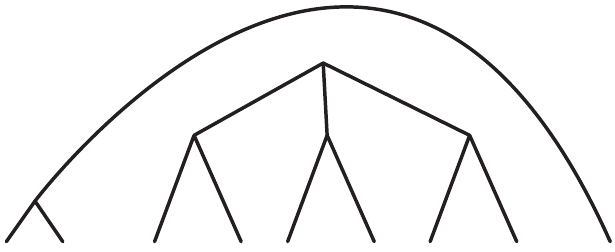}}& $s_1,s_3,s_5,s_7$ \\
  &&\\
\hline
&&\\

 {\large $\fw_{1,2}(W_1)$} &\raisebox{-15pt}{\includegraphics[width=1in]{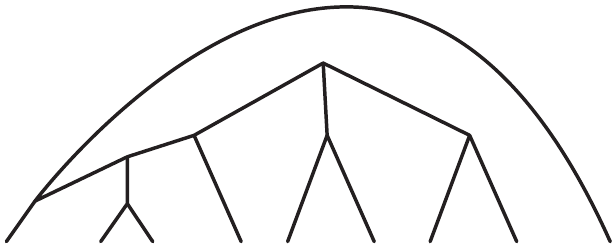}}& $s_2,s_5,s_7$ \\
  &&\\
\hline
&&\\
	
 {\large $W_2$} &\raisebox{-15pt}{\includegraphics[width=1in]{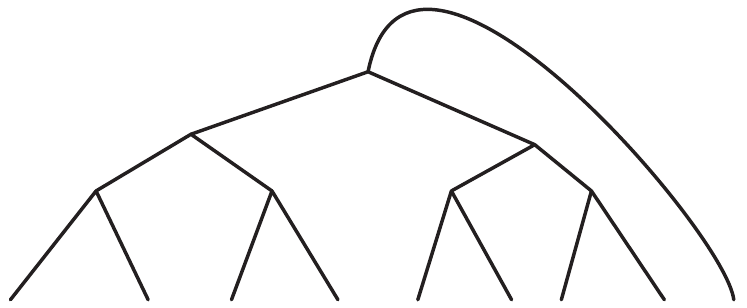}}& $s_1,s_3,s_5,s_7$ \\
  &&\\
\hline 
&&\\

 {\large $\fw_{1,2}(W_2)$} &\raisebox{-15pt}{\includegraphics[width=1in]{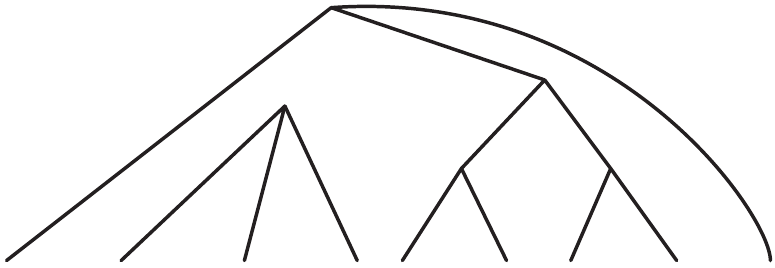}}& $s_2,s_3,s_5,s_7$ \\
  &&\\
\hline

&&\\

 {\large $W_3$} &\raisebox{-15pt}{\includegraphics[width=1in]{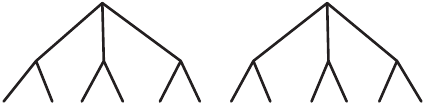}}& $s_1,s_3,s_5,s_7,s_9,s_{11}$ \\
  &&\\
\hline

&&\\

 {\large $\fw_{7,6}(W_3)$} &\raisebox{-15pt}{\includegraphics[width=1in]{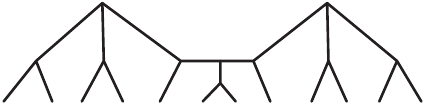}}& $s_1,s_3,s_6,s_9,s_{11}$ \\
  &&\\
\hline

&&\\

 {\large $\fw_{1,2}(\fw_{7,6}(W_3))$} &\raisebox{-15pt}{\includegraphics[width=1in]{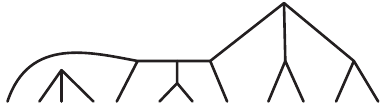}}& $s_2,s_3,s_6,s_9,s_{11}$ \\
  &&\\
\hline
 \end{tabular}
\caption{Webs for examples}
\label{examplewebs}
\end{figure}

\begin{figure}[ht]
\begin{tabular}{|c|c|c|}
\hline&&\\ 
& {Tableau} &{\large $\tau$} \\
&&\\
\hline &&\\

 {\large $T_1$} &\raisebox{-15pt}{\young(135,247,689)}& $s_1,s_3,s_5,s_7$ \\
  &&\\
\hline 

&&\\
 {\large $\fyt_{1,2}(T_1)$} &\raisebox{-15pt}{\young(125,347,689)}& $s_2,s_5,s_7$ \\
  &&\\
\hline
&&\\
 {\large $T_2$} &\raisebox{-15pt}{\young(1379,258\eleven,46\ten\twelve)}&$s_1,s_3,s_5,s_7,s_9,s_{11}$  \\
  &&\\
\hline
&&\\
 {\large $T_3$} &\raisebox{-15pt}{\young(1357,269\eleven,48\ten\twelve)}&  $s_1,s_3,s_5,s_7,s_9,s_{11}$\\
  &&\\
\hline
&&\\
 {\large $T_4$} &\raisebox{-15pt}{\young(1357,249\eleven,68\ten\twelve)}&  $s_1,s_3,s_5,s_7,s_9,s_{11}$\\
  &&\\
\hline
&&\\
 {\large $T_5$} &\raisebox{-15pt}{\young(1359,247\eleven,68\ten\twelve)}&  $s_1,s_3,s_5,s_7,s_9,s_{11}$\\
  &&\\
\hline
\end{tabular}
\caption{Tableaux for examples, part I.}
\label{exampletableaux1}
\end{figure}

\begin{figure}[ht]
\begin{tabular}{|c|c|c|}
\hline&&\\ 
& {Tableau} &{\large $\tau$} \\
&&\\
\hline &&\\
&&\\
 {\large $\fyt_{7,6}(T_2)$} &\raisebox{-15pt}{\young(1369,258\eleven,47\ten\twelve)}&$s_1,s_3,s_6,s_9,s_{11}$  \\
  &&\\
  \hline
  &&\\
 {\large $\fyt_{7,6}(T_3)$} &\raisebox{-15pt}{\young(1358,269\eleven,47\ten\twelve)}&  $s_1,s_3,s_5,s_6,s_8,s_9,s_{11}$\\
  &&\\
\hline
&&\\
 {\large $\fyt_{7,6}(T_4)$} &\raisebox{-15pt}{\young(1356,249\eleven,78\ten\twelve)}&  $s_1,s_3,s_6,s_9,s_{11}$\\
  &&\\
\hline
&&\\
 {\large $\fyt_{7,6}(T_5)$} &\raisebox{-15pt}{\young(1359,246\eleven,78\ten\twelve)}&  $s_1,s_3,s_5,s_6,s_9,s_{11}$\\
  &&\\
\hline
&&\\
 {\large $\fyt_{1,2}(\fyt_{7,6}(T_2))$} &\raisebox{-15pt}{\young(1269,358\eleven,47\ten\twelve)}&$s_2,s_3,s_6,s_9,s_{11}$  \\
  &&\\
  \hline
 &&\\
 {\large $\fyt_{1,2}(\fyt_{7,6}(T_4))$} &\raisebox{-15pt}{\young(1256,349\eleven,78\ten\twelve)}&  $s_2,s_6,s_9,s_{11}$\\
  &&\\
\hline
\end{tabular}
\caption{Tableaux for examples, part II.}
\label{exampletableaux2}
\end{figure}

\section{\texorpdfstring{$\la{sl}_3$}{sl3}-webs and \kl Theory}\label{webandkl}

The braid group maps naturally to the Hecke algebra, allowing us to relate $\la{sl}_3$-webs to the \kl representations  described earlier.  In this section, we discuss a presentation  $\dot{\mathscr{H}}_{3n}$ of the Hecke algebra that is  compatible with $\la{sl}_3$-webs, together with the map from the braid group to $\dot{\mathscr{H}}_{3n}$.  We also prove that---despite sharing the same generalized $\tau$-invariants, as per Section \ref{bijectionandtau}---the \kl left cell basis is not the same as the basis of reduced webs in this case.  Our proof uses a counterexample similar to Khovanov and Kuperberg's proof that web bases are not dual canonical \cite{KK}.  To the best of our knowledge, this is the only formal proof of a result that some expected to hold by Schur-Weyl duality.

We first describe a presentation $\dot{\mathscr{H}}_{3n}$ of the Hecke algebra that fits better with Khovanov's conventions for $\la{sl}_3$-webs \cite{MK}, including using the variable $q$ rather than $v$. Let $R=\bC[q,q\inv]$ and denote the group algebra of the braid group over $R$ by $RB_{3n}$.  Denote the generators of $\dot{\mathscr{H}}_{3n}$ by $\dot{T}_{s_i}$. Replace the Hecke algebra relation shown in Equation \ref{heckerelation} with 
\begin{equation}\label{newheckerelation}
(\dT{i}-q^2)(\dT{i}+q^4)=0.
\end{equation}
The map $RB_{3n} \rightarrow \dot{\mathscr{H}}_{3n}$ sends the positive crossing in Equation \ref{poscrosseq} to $\dT{i}$. Equivalently, we obtain $\dot{\mathscr{H}}_{3n}$ from $RB_{3n}$ by imposing the skein relation in 
 Figure \ref{braidgroupquotientrelation}. 
 \begin{figure}[h]

\begin{tikzpicture}[baseline=.25cm, scale=.25]
\node at (-11, 2) {$q^3$};
\draw[style= thick, ->] (-8,0)--(-10,4);
\draw[style=thick,] (-10,0) -- (-9.25, 1.5);
\draw[style=thick, ->] (-8.75, 2.5) -- (-8,4);
\node at (-5,2) {$-q^{-3}$};
\draw[style= thick, ->] (-3,0)--(-1,4);
\draw[style=thick,] (-1,0) -- (-1.75, 1.5);
\draw[style=thick, ->] (-2.25, 2.5) -- (-3,4);
\node at (4.5, 2) {$= (q - q^{-1})$};
\draw [style=thick, ->] (9,0) to[out=45, in=270] (10,2);
\draw [style=thick] (10,2) to[out=90, in=110] (9,3.75);
\draw [style=thick, ->] (12,0) to[out=135, in=270] (11,2);
\draw [style=thick] (11,2) to[out=90, in=70] (12,3.75);

\end{tikzpicture}

\caption{The relation carrying the group algebra $RB_{3n}$ to $\dot{\mathscr{H}}_{3n}$.}
\label{braidgroupquotientrelation}
\end{figure}
 
The algebra $\dot{\mathscr{H}}_{3n}$ is in fact isomorphic to $\mathscr{H}_{3n}$. We sketch the isomorphism here; Bigelow gives more details \cite{MR2264547}. Let $\ol{A}=\bC[v^{1/2},v^{-1/2}]$ be the complexification of the ring $A$ from Section \ref{klsection}. Define an isomorphism $\ol{A} \rightarrow R$ by taking $v^{1/2}$ to $-1/q$. The isomorphism $\mathscr{H}_{3n} \rightarrow \dot{\mathscr{H}}_{3n}$ is then defined by taking $\T{i}$ to $(1/q^4)\dT{i}$. Figure \ref{Tidiagram} shows the action of $\T{i}$ on webs in diagrammatic form. (The astute reader may notice a parallel with Equation \ref{klaction}.  This parallel is reinforced in  Lemma \ref{taurulelemma}, where we identify webs on which $\T{i}$ acts by $-1$.)
\begin{figure}[ht]

\begin{tikzpicture}[baseline=.25cm, scale=.25]
\node at (-2,2) {$T_{s_i} \cdot$};
\node at (0, -1) {\tiny{$i$}};
\node at (2, -1) {\tiny{$i+1$}};
\draw[style= thick, ->] (0,0)--(0,4);
\draw[style=thick, ->] (2,0) -- (2,4);
\node at (5, 2) {$= v$};
\draw [style=thick, ->] (7,0) to[out=45, in=270] (8,2);
\draw [style=thick] (8,2) to[out=90, in=110] (7,3.75);
\draw [style=thick, ->] (10,0) to[out=135, in=270] (9,2);
\draw [style=thick] (9,2) to[out=90, in=70] (10,3.75);
\node at (12.5, 2.25) {$+v^{1/2}$};
\draw [style = thick, ->] (15,0) -- (15.5, .5);
\draw [style = thick] (15.5,.5) -- (16, 1);
\draw [style = thick, ->] (17,0) -- (16.5, .5);
\draw [style = thick] (16.5,.5) -- (16, 1);
\draw[style=thick,](16, 1) -- (16, 2);
\draw[style=thick, <-](16, 2) -- (16, 3);
\draw [style = thick, ->] (15,4) -- (15.5, 3.5);
\draw [style = thick] (15.5,3.5) -- (16, 3);
\draw [style = thick, ->] (17,4) -- (16.5, 3.5);
\draw [style = thick] (16.5,3.5) -- (16, 3);

\end{tikzpicture}

\caption{The diagrammatic action of $\T{i}$ on webs.}\label{Tidiagram}
\end{figure}

Suppose that $Q$ is a standard tableau of shape $[n,n,n]$ and let $\CC_Q$ be the corresponding \kl left cell, namely the set $\CC_Q=\{w\in S_{3n}\vert Q(w)=Q\}$. Theorem \ref{RSleftcell} stated that $\operatorname{KL}_{\CC_Q}$ is isomorphic as a complex symmetric group representation to the $[n,n,n]$-representation of $S_{3n}$. Results of Murphy show that $W_{3n}$ and $\operatorname{KL}_{\CC_Q}$ are isomorphic as $\mathscr{H}_{3n}$-modules over the complex numbers \cite{MR1327362}. 

A natural conjecture is that the web basis and \kl basis are equivalent.   In fact, these bases differ whenever $n$ is greater than $5$, as we show in the next theorem via a counterexample. 

\begin{theorem}\label{thm:Inequivalence of bases}
The web basis for $W_{3n}$ differs from the Kazhdan-Lusztig basis when $n \geq 6$.
\end{theorem}

\begin{proof}
Figures \ref{counterexample1} and \ref{counterexample2} give two webs, drawn so the base vertices are on a circle. As usual, set $q=-1$ in the Hecke algebra.  

Consider the multiplicity of $W'$ in the reduced-web decomposition of $s_1\cdot W$.  We  show that a unique sequence of reductions leads to $W'$.  Use Figure \ref{braidingmorphisms} to write $s_1 \cdot W$ as the sum of the reduced web $W$ and an unreduced web with a unique square.  Next, apply the square relation in Equation \eqref{squarerel} repeatedly; at each step, one term reduces to a web with fewer connected components than $W'$ and  a second term  peels off a new connected component of $W'$.  The resulting web looks like $W'$ except one connected component has a bigon.  The bigon relation in Equation \eqref{bigonrel} shows that $W'$ appears in $s_1\cdot W$ with multiplicity $-2$. 

This is impossible in the \kl basis since all coefficients $\mu$ in the \kl graph for $S_n$ are nonnegative  \cite{MR560412,MR632980,MR610137}. 

When $n>6$, we use the same argument on the web consisting of the union of $W$ with any desired number of copies of the connected components of $W'$.
\end{proof}

\begin{figure}[ht]
\begin{picture}(100,80)(0,12)
\put(0,50){
$W=
\raisebox{-.6in}{\includegraphics[height=1.2in]{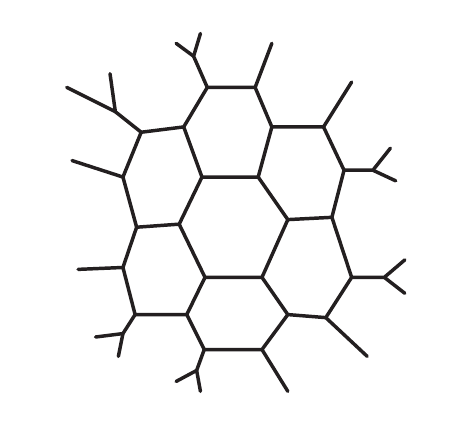}}$}
\put(35,73){1}
\put(35,58){2}
\put(37,35){3}
\end{picture}
\caption{Example showing that the web basis differs from the \kl basis, part I.}\label{counterexample1}
\end{figure}
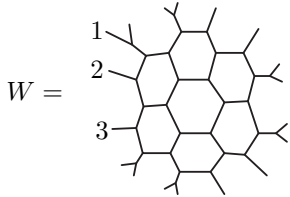
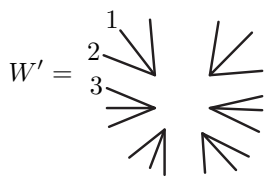
\begin{figure}[ht]\begin{picture}(100,72)(0,8)
\put(0,50){
$W'=
\raisebox{-.5in}{\includegraphics[height=.9in]{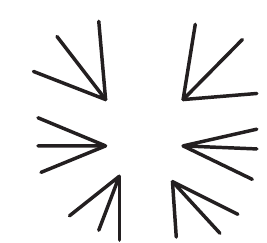}}$}
\put(40,70){1}
\put(33,58){2}
\put(34,45){3}
\end{picture}
\caption{Example showing that the web basis differs from the \kl basis, part II.}\label{counterexample2}
\end{figure}

\section{Acknowledgements}

We are grateful to Monty McGovern, John Stembridge and Peter Trapa for helpful discussions.

\def\cprime{$'$}

\end{document}